\documentclass[11pt,a4paper]{amsart}

\usepackage{amsfonts}
\usepackage[leqno]{amsmath}
\usepackage[noadjust]{cite}
\usepackage{amssymb}
\usepackage{amstext}
\usepackage{caption}
\usepackage{subcaption}
\usepackage{graphicx}
\DeclareGraphicsExtensions{eps,jpg,tif}
\usepackage{titletoc}
\usepackage[english]{babel}
\usepackage{lineno}
\usepackage{calc}
\usepackage{mathabx}
\usepackage{pgf,tikz}
\usepackage[colorlinks=true,linkcolor=blue,pdftex=true,urlcolor=blue,hyperindex=true]{hyperref}	
\usepackage[nameinlink]{cleveref}
\usepackage[T1]{fontenc}
\usepackage{comment}

\newtheorem{theorem}{Theorem}[section]
\newtheorem{assumption}[theorem]{Assumption}

\newtheorem{corollary}[theorem]{Corollary}

\newtheorem{lemma}[theorem]{Lemma}

\newtheorem{proposition}[theorem]{Proposition}

\theoremstyle{definition}
\newtheorem{definition}[theorem]{Definition}
\newtheorem{example}[theorem]{Example}
\newtheorem{notation}[theorem]{Notation}
\newtheorem{remark}[theorem]{Remark}

\newcommand{\MNN}[4]{\left\Vert #1\right\Vert_{#2,#3,#4}}
\newcommand{\m}[1]{\left\vert #1\right\vert}
\newcommand{\DCN}{D_{\rho_0,\rho_1,...,\rho_N}}
\newcommand{\DCn}{D_{\rho_1,...,\rho_N}}
\newcommand{\DCr}{D_{r,...,r}}
\newcommand{\Domn}{\mathcal{O}(D_{\rho_1,...,\rho_N})[[t]]}
\newcommand{\DomGn}[1]{\mathcal{O}(D_{\rho_1,...,\rho_N})[[t]]_{#1}}

\newcommand{\domn}{\mathcal{O}(D_{\rho_1,...,\rho_N})}
\newcommand{\TC}[2]{#1_{#2,\ast}}
\newcommand{\G}[1]{\Gamma(1+(s+1)#1)}

\newcommand{\C}{\mathbb C}
\newcommand{\R}{\mathbb R}
\newcommand{\Na}{\mathbb N}
\newcommand{\dprod}[3]{\displaystyle\prod_{#1}^{#2}#3}
\newcommand{\dsum}[3]{\displaystyle\sum_{#1}^{#2}#3}
\newcommand{\dint}[3]{\displaystyle\int_{#1}^{#2}#3}

\newcommand{\Oo}{\mathcal{O}}

\begin{document}

\title[Gevrey regularity for inhomogeneous nonlinear moment PDEs]{On Gevrey regularity of solutions for inhomogeneous nonlinear moment partial differential equations}

\author{Pascal Remy}
\address{Laboratoire de Mathématiques de Versailles,\\ Université de Versailles Saint-Quentin,\\ 45 avenue des
Etats-Unis, 78035 Versailles cedex, France}
\email{pascal.remy@uvsq.fr}

\author{Maria Suwi\'nska}
\address{Faculty of Mathematics and Natural Sciences,
College of Science\\
Cardinal Stefan Wyszy\'nski University\\
W\'oycickiego 1/3,
01-938 Warszawa, Poland}
\email{m.suwinska@op.pl}
\date{}

\maketitle

\begin{abstract}
 In this article we investigate Gevrey regularity of formal power series solutions for a certain class of nonlinear moment partial differential equations, the inhomogeneity of which is $\sigma$-Gevrey with respect to the time variable $t$ for a fixed $\sigma\ge 0$. The results are achieved by analyzing the geometric structure of the Newton polygon associated with the equation and are a generalization of similar results obtained for standard nonlinear partial differential equations as well as linear moment differential equations.
\end{abstract}

\section{Introduction}
The topic of Gevrey regularity has been studied freqeuntly in recent years. In particular, many advances have been made concerning formal solutions of linear partial differential equations with notable works being, among many others, \cite{BLR,BalYosh10,R16,R17,R20}. The topic has also been considered for nonlinear partial differential equations, as can be seen for example in \cite{Tah11,Tah12,ASh,R23a,R23b}.

Even more recently the notions of Gevrey estimates and summability have been applied to linear moment differential equations in \cite{Mich17,MichSu20,Su21} as well as their generalizations in the framework of strongly regular sequences in \cite{LaMichSu1,LaMichSu2,LaMichSu3,LaMichSu4}.

The purpose of the present work is to combine the results obtained by both authors regarding nonlinear partial differential equations and linear moment differential equatins. More precisely, we aim to generalize the results from \cite{R23a,R23b} to the case of nonlinear moment partial differential equations by using methods applied previously in \cite{Su21} exclusively to linear moment differential equations.

In the present paper, we consider a class of nonlinear moment partial differential equations in $1$-dimensional time variable $t\in\C$ and $N$-dimensional spatial variable $x=(x_1,...,x_N)\in\C^N$ of the form 
\begin{equation}\label{eqn}
\begin{cases}
\partial_{m_0;t}^\kappa u - P(t,x,(\partial_{m_0;t}^i\partial_{m;x}^qu)_{(i,q)\in\Lambda})=\widetilde{f}(t,x)\\
\partial_{m_0;t}^ju(t,x)_{|t=0}=\varphi_j(x) \textrm{ for }0\le j<\kappa,
\end{cases}
\end{equation}
where $P$ is a polynomial with analytic coefficients on a polydisc $D_{\rho_0,\rho_1,...,\rho_N}:=D_{\rho_0}\times D_{\rho_1}\times...\times D_{\rho_N}$ centered at the origin of $\C^{N+1}$ ($D_\rho$ stands for the disc with center $0\in\C$ and radius $\rho>0$), the inhomogeneity $\widetilde{f}(t,x)$ is a formal power series with respect to $t$ with all coefficients analytic on tbe polydisc $D_{\rho_1,...,\rho_N}$, and where the initial data $\varphi_j(x)$ are all analytic on $\DCn$. Our aim is to show that the Gevrey regularity of the formal solution $\widetilde{u}(t,x)$ of Eq. (\ref{eqn}) depends both on the Gevrey regularity of the inhomogeneity $\widetilde{f}(t,x)$ and on the structure of Eq. (\ref{eqn}), that is on the nonlinear operator $\Delta_{\kappa,P}:=\partial_{m_0;t}^\kappa-P(t,x,(\partial_{m_0;t}^i\partial_{m;x}^q)_{(i,q)\in\Lambda})$. More precisely, if $\widetilde{f}(t,x)$ is $\sigma$-Gevrey for a certain $\sigma\geq 0$, then $\widetilde{u}(t,x)$ is of Gevrey order either $\sigma$ or $\sigma_c>\sigma$ with $\sigma_c>0$ a nonnegative real number entirely determined by the operator $\Delta_{\kappa,P}$.

The paper is structured as follows:

In Section \ref{section:moments}, definitions of moment functions and moment differential operators are recalled with some basic properties listed. Regular moment functions are also defined. For more details on kernel functions and their associated moment functions as well as various operators connected to them we refer the reader to \cite{Balser}. In Section \ref{section:gevrey-series}, the definition and various properties of Gevrey formal power series are given. After that, in Section \ref{section:nagumo}, the notion of modified Nagumo norms is fleshed out, generalizing slightly the results shown in \cite{Su21}. In particular we prove that the norm defined in Definition \ref{defi:ModNagumoNorm} has properties analogous to classical Nagumo norm.

The main problem considered in this paper is properly introduced in Section \ref{section:main}. A definition of the Newton polygon for the considered equation is proposed (Definition \ref{defi:NewtonPolygon}). We also show that the problem is \textit{formally well-posed}, that is the considered equation has a unique formal power series solution under given assumptions. The main result of the paper is presented in Theorem \ref{GevreyRegularityTheorem}, which connects the critical value of the equation and the Gevrey order of the inhomogeneity with the Gevrey order of its formal solution. At the end of this section we also introduce several examples showcasing this result.

The last Section \ref{section:proof} is devoted entirely to the proof of Theorem \ref{GevreyRegularityTheorem}. First we use the modified Nagumo norms and the majorant method to prove the first point of the theorem. To prove the second point, we present a detailed example similar to the one used in a similar manner in \cite{R23b}.

{\begin{notation}
Throughout this paper, we use the following notations:
\begin{itemize}
\item $\Na$ stands for the set of all nonnegative integers and $\Na^*=\Na\setminus\{0\}$ for the set of all positive integers.
\item $\R^+$ stands for  the set of all the nonnegative real numbers and $\R^*_+$ for the set of all the positive real numbers.
\item For any $\alpha=(\alpha_1,\ldots,\alpha_N)\in(\R^+)^N$, we use $\lambda(\alpha)$ to denote the sum $\alpha_1+\ldots+\alpha_N$.
\item For any $\alpha=(\alpha_1,\ldots,\alpha_N)\in(\R^+)^N$, $\beta=(\beta_1,\ldots,\beta_N)\in(\R^+)^N$ and $c\in\R^+$, we use the following classical operations:
\begin{itemize}
\item$\alpha+\beta=(\alpha_1+\beta_1,...,\alpha_N+\beta_N)$;
\item$c\alpha=(c\alpha_1,...,c\alpha_N)$;
\item$\alpha\beta=(\alpha_1\beta_1,...,\alpha_N\beta_N)$ so that $\lambda(\alpha\beta)$ coincides with the usual scalar product in $\R^N$ between $\alpha$ and $\beta$;
\end{itemize}
\item for any $q=(q_1,\ldots,q_N)\in\Na^N$, $x=(x_1,\ldots,x_N)\in\R^N$ and moment functions $m_1,\ldots,m_N$, we use the following classical notation for moment differential operators: $\partial_{m;x}^q=\partial_{m_1;x_1}^{q_1}...\partial_{m_N;x_N}^{q_N}$.
\item $\Gamma$ stands for the Gamma function and $\Psi=(\ln\Gamma)'=\Gamma'/\Gamma$ for the Psi function.
\item For any $\rho_1,\ldots,\rho_N>0$ we denote by $\DCn$ the polydisc $D_{\rho_1}\times\ldots\times D_{\rho_N}\subset\C^N$, where $D_\rho=\{z\in\C:\ |z|<\rho\}$ for any $\rho>0$.
\item For any $d\in\R$ and $\alpha,R>0$, an open sector in direction $d$ with an opening $\alpha$ and a radius $R$ is a set
$$
S_d(\alpha,R)=\left\{x\in\C:\ 0<|x|<R,\,|\arg x-d|<\frac{\alpha}{2}\right\}.
$$
For a sector of an infinite radius we will use a notation $S_d(\alpha)$. Whenever the opening is not relevant, it will be omitted in the notation.
\item Given any open set $U\subset \C^N$, $N\in\Na^*$, we denote by $\Oo(U)$ the set of all holomorphic functions defined in $U$. The set of all formal power series in variable $t$ with coefficients from a fixed nonempty set $F$ will be denoted by $F[[t]]$. Similarly, by $\Oo[[t]]$ we will denote the set of all formal power series in variable $t$ with analytic coefficients in some common neighborhood of the origin.
\end{itemize}
\end{notation}

\section{Moment functions and moment differential operators}\label{section:moments}
\subsection{Moment functions}

Below we present the classical approach to kernel functions and their corresponding moment functions as given in \cite{Balser}.

\begin{definition}\label{defi:kernel-moment}
A pair $(e,E)$ of $\C$-valued functions is called \textit{kernel functions of order $s<2$} if the three following conditions hold:
\begin{enumerate}
\item The function $e$ satisfies the following points:
\begin{enumerate}
\item$e$ is holomorphic on the sector $S_{0}(\pi s)$;
\item$e(t)>0$ for all $t>0$;
\item the function $t^{-1}e(t)$ is integrable at zero;
\item $e$ is $k$-exponentially flat at infinity for $k=1/s$, that is, for every $\varepsilon>0$, there exist two positive constants $A,B>0$ such that $|e(x)|\le A\exp(-(|x|/B)^k)$ for all $x\in S_0(\pi s-\varepsilon)$.
\end{enumerate}
\item The function $E$ satisfies the following points:
\begin{enumerate}
\item$E$ is entire on $\C$ with a global exponential growth of order at most $k=1/s$ at infinity;
\item the function $t^{-1}E(t)$ is integrable at zero in $S_{\pi}(\pi(2-s))$.
\end{enumerate}
\item The functions $e$ and $E$ are connected by a corresponding \textit{moment function $m$ of order $s$} as follows:
\begin{enumerate}
\item the function $m$ is defined by the Mellin transform of $e$:
\begin{equation}\label{DefMomentFunction}
m(\lambda)=\dint{0}{+\infty}{t^{\lambda-1}e(t)dt}\quad\text{for all Re}(\lambda)\geq0;
\end{equation}
\item the function $E$ has the power series expansion
\begin{equation}\label{LinkEMomentFunction}
E(t)=\dsum{j\geq0}{}{\dfrac{t^j}{m(j)}}\quad\text{for all }t\in\C.
\end{equation}
\end{enumerate}
\end{enumerate}
\end{definition}

\begin{remark}
 For the sake of simplicity, we shall henceforth assume that $m(0)=1$ for any moment function.
\end{remark}

\begin{definition}[Moment sequence]
 Let us consider a moment function $m$ of order $s$. Then we call $(m(j))_{j\ge 0}$ a moment sequence of order $s$.
\end{definition}

It is necessary to adjust Definition \ref{defi:kernel-moment} so that kernel functions of all positive orders $s\ge 2$ can be considered as well.

\begin{definition}[See \cite{Balser}, Section 5.6.]
 Let $s>0$ and suppose that there exists $p\in\Na$ such that $s/p<2$. Then we define a kernel function
$e$ of order $s$ as 
\begin{equation*}
e(t)=\frac{\hat{e}(x^{\frac{1}{p}})}{p}.
\end{equation*}
where $\hat{e}(t)$ is a kernel function of order $s/p<2$ as defined in Definition \ref{defi:kernel-moment}. Then the corresponding kernel function $E(t)$ and moment function $m$ are defined by the same formul\ae\ in relation to $e(t)$ and each other as in Definition \ref{defi:kernel-moment}.
\end{definition}

\begin{example}
The following classical example of kernel functions and their corresponding moment function is widely used in the classical theory of $k$-summability:
\begin{itemize}
 \item $e(t)=kt^{k}e^{-t^k}$;
  \item $E(t)=\dsum{j\ge 0}{}{\frac{t^j}{\Gamma(1+sj)}}=\mathbf{E}_{s}(x)$ the Mittag-Leffler function of index $s$;
  \item $m(\lambda)=\Gamma(1+s\lambda)$.
\end{itemize}
\end{example}

\begin{proposition}[See \cite{Balser}, Section 5.5.]
Observe that the integral (\ref{DefMomentFunction}) being absolutely and locally uniformly convergent, the function $m$ is holomorphic for $\text{Re}(\lambda)>0$ and continuous up to the imaginary axis, and the values $m(\lambda)$ are positive real numbers for all $\lambda\geq0$. Moreover, accordingly the asymptotic behavior of kernel functions $e$ and $E$, we deduce from the identities (\ref{DefMomentFunction}) and (\ref{LinkEMomentFunction}) that there exist four positive constants $c,C,a,A>0$ such that the following estimate holds for all $j\geq0$:
\begin{equation}\label{IneqMomentFunction}
ca^j\G{j}\leq m(j)\leq CA^j\G{j}.
\end{equation}
\end{proposition}

The concept of regular moment functions described below was first introduced in \cite{MichSu20}. It was also used later,   without the connection to kernel functions, in \cite{Su21}.

\begin{definition}[Regular moment function]\label{defi:moment_regular}
 A moment function $m$ of order $s>0$ is called \textit{regular} if there exist two positive constants $a,A>0$ such that
 $$
 a(j+1)^s\le\frac{m(j+1)}{m(j)}\le A(j+1)^s\quad\textrm{for every}\quad j\in\Na.
 $$
\end{definition}

\begin{example}
For any fixed $s>0$, the moment function $m(\lambda)=\Gamma(1+s\lambda)$ is a regular moment function of order $s$.

Indeed, if we consider Stirling's Formula
\begin{equation*}
\label{eq:stirling}
\sqrt{2\pi}t^{t-\frac{1}{2}}e^{-t}\leq\Gamma(t)\leq\sqrt{2\pi}e^{\frac{1}{12t}}t^{t-\frac{1}{2}}e^{-t}<\sqrt{2\pi}t^{t-\frac{1}{2}}e^{-t+1}\textrm{ for every }t\geq 1,
\end{equation*}
then for every $j\in\Na$
$$
\frac{\Gamma(1+js)}{\Gamma(1+js-s)} \le e^{-s+1}\left(\frac{1+js}{1+js-s}\right)^{1+js-s-\frac{1}{2}}(1+js)^{s}\le \left(1+\frac{1}{s}\right)^{s}es^{s}j^{s}
$$
and
$$
\frac{\Gamma(1+js)}{\Gamma(1+js-s)} \ge e^{-s-1}\left(\frac{1+js}{1+js-s}\right)^{js-s+\frac{1}{2}}(1+js)^{s}\ge e^{-s-1}s^{s}j^{s}.
$$
\end{example}

\subsection{Moment differentiation}
The notion of moment differential operators or moment derivatives was first introduced by W.~Balser and M.~Yoshino in \cite{BalYosh10}.

\begin{definition}[Moment derivation]\label{MomentDifferentiationOperator}
Let $m_0$ be a moment function of order $s_0>0$ and $\widetilde{u}(t,x)\in\Domn$ a formal power series written in the form
$$\widetilde{u}(t,x)=\sum_{j\geq0}u_{j,*}(x)\frac{t^{j}}{m_0(j)}.$$
Then, the \textit{moment derivative $\partial_{m_0;t}\widetilde{u}$ of $\widetilde{u}(t,x)$ with respect to $t$} is the formal power series in $\Domn$ defined by
$$\partial_{m_0;t}\widetilde{u}(t,x)=\dsum{j\geq0}{}{\TC{u}{j+1}(x)\dfrac{t^j}{m_0(j)}}.$$
\end{definition}

Observe that, for $m_0(\lambda)=\Gamma(1+\lambda)$, the operator $\partial_{m_0;t}$ coincides with the standard derivation operator $\partial_t$ with respect to $t$.

Observe also that Definition \ref{MomentDifferentiationOperator} can be naturally extended to analytic functions at the origin of $\C^{n+1}$ by means of their representation in the form of an infinite series. In particular, we can define in the same way the moment derivation $\partial_{m_j;x_j}$ with respect to $x_j$ for any moment function $m_j$ of order $s_j>0$ and any $j\in\{1,...,N\}$. Thereby, for any formal power series $\widetilde{u}(t,x)\in\Domn$ written in the form
$$\widetilde{u}(t,x)=\dsum{j_0,j_1,...,j_N\geq0}{}{u_{j_0,j_1,...,j_N}\dfrac{t^{j_0}}{m_0(j_0)}\dfrac{x_1^{j_1}}{m_1(j_1)}...\dfrac{x_N^{j_N}}{m_N(j_N)}},$$
the following identity holds for any $i_0,i_1,...,i_N\geq0$:
\begin{multline*}
\partial_{m_0;t}^{i_0}\partial_{m_1;x_1}^{i_1}...\partial_{m_N;x_N}^{i_N}\widetilde{u}(t,x)\\=\dsum{j_0,...,j_n\geq0}{}{u_{j_0+i_0,j_1+i_1,...,j_N+i_N}\dfrac{t^{j_0}}{m_0(j_0)}\dfrac{x_1^{j_1}}{m_1(j_1)}...\dfrac{x_N^{j_N}}{m_N(j_N)}}.
\end{multline*}

Observe that the operator $\partial_{m_0;t}$ commutes with any operator $\partial_{m_j;x_j}$, and that the operator $\partial_{m_j;x_j}$ commutes with any operator $\partial_{m_\ell;x_\ell}$ as soon as $j\neq\ell$.

Observe also that the previous definition can be also naturally extended to analytic functions at the origin of $\C^{N+1}$ by means of their representation in the form of an infinite series. Doing so, and using inequality (\ref{IneqMomentFunction}), one can easily check that, if $a(t,x)$ is an analytic function at the origin of $\C^{N+1}$, say on a polydisc $\DCN$, then the formal power series $\partial_{m_0;t}^{i_0}\partial_{m_1;x_1}^{i_1}...\partial_{m_N;x_N}^{i_N}a(t,x)$ defines an analytic function on a polydisc $D_{\rho_0',\rho_1',...,\rho_N'}$ with convenient radii $0<\rho_j'\leq\rho_j$ for all $j=0,...,N$. In particular, this function may be analytic on a polydisc smaller than the initial polydisc of analyticity of $a(t,x)$.

However, as the following result shows, this does not occur in the case where the moments $m_j$ are all regular.

\begin{proposition}\label{AnalyticFunctionMomentDerivatives}
Let $a(t,x)\in\mathcal O(\DCn)$ be an analytic function on $\DCn$, and let $m_0,m_1,...,m_n$ be $n+1$ regular moment functions of respective orders $s_0,s_1,...,s_n>0$. Then, for any $i_0,i_1,...,i_N\geq0$, the formal power series $\partial_{m_0;t}^{i_0}\partial_{m_1;x_1}^{i_1}...\partial_{m_N;x_N}^{i_n}a(t,x)$ also define analytic functions on $\DCn$.
\end{proposition}

\begin{proof}
Let us write $a(t,x)$ in the form
$$a(t,x)=\dsum{j_0,j_1,...,j_N\geq0}{}{a_{j_0,j_1,...,j_N}\dfrac{t^{j_0}}{m_0(j_0)}\dfrac{x_1^{j_1}}{m_1(j_1)}...\dfrac{x_N^{j_N}}{m_N(j_N)}}$$
so that
$$\partial_{m_0;t}^{i_0}\partial_{m_1;x_1}^{i_1}...\partial_{m_N;x_N}^{i_N}a(t,x)=\dsum{j_0,...,j_N\geq0}{}{v_{j_0,j_1,...,j_N}t^{j_0}x_1^{j_1}...x_N^{j_N}}$$
with
$$v_{j_0,j_1,...,j_N}=\dfrac{a_{j_0+i_0,j_1+i_1,...,j_N+i_N}}{m_0(j_0)m_1(j_1)...m_N(j_N)}.$$
For any $d\in\{0,...,N\}$, let us choose two radii $r_d,r'_d>0$ such that $r_d<r'_d<\rho_d$. By assumption, there exists a positive constant $C
>0$ such that
\begin{equation*}
\m{\dfrac{a_{j_0,j_1,...,j_N}}{m_0(j_0)m_1(j_1)...m_N(j_N)}}\leq C\left(\dfrac{1}{r_0'}\right)^{j_0}\left(\dfrac{1}{r_1'}\right)^{j_1}...\left(\dfrac{1}{r_N'}\right)^{j_N}
\end{equation*}
for all $j_0,j_1,...,j_N\geq0$. Then, for all $\m{t}\leq r_0$ and all $\m{x_d}\leq r_d$, $d=1,...,N$, we get
$$\m{v_{j_0,j_1,...,j_N}t^{j_0}x_1^{j_1}...x_N^{j_N}}\leq\dfrac{C}{r_0'^{i_0}r_1'^{i_1}...r_N'^{i_N}}\left(\dprod{d=0}{N}{\left(\dfrac{r_d}{r'_d}\right)^{j_d}}\right)\left(\dprod{d=0}{N}{\dfrac{m_d(j_d+i_d)}{m_d(j_d)}}\right)$$
for all $j_0,j_1,...,j_N$. Since $m_d$ is a regular moment function of order $s_d$, there exist two positive constants $c_d,C_d>0$ such that
$$c_d (j+1)^{s_d}\leq \dfrac{m_d(j+1)}{m_d(j)}\leq C_d (j+1)^{s_d}$$
for all $j\geq0$. Then,
$$\dfrac{m_d(j_d+i_d)}{m_d(j_d)}=\begin{cases}
1&\text{if }i_d=0\\
\dprod{k=j_d}{j_d+i_d-1}{\dfrac{m_d(k+1)}{m_d(k)}}&\text{if }i_d\geq1\end{cases}\leq C_{d}^{i_d}(j_d+1)^{s_d}...(j_d+i_d)^{s_d}$$
and the previous estimates become
$$\m{v_{j_0,j_1,...,j_n}t^{j_0}x_1^{j_1}...x_n^{j_n}}\leq\dfrac{C }{r_0'^{i_0}r_1'^{i_1}...r_n'^{i_n}}\left(\dprod{d=0}{N}{C_d^{i_d}(j_d+1)^{s_d}...(j_d+i_d)^{s_d}\left(\dfrac{r_d}{r'_d}\right)^{j_d}}\right)$$
for all $j_0,j_1,...,j_N$, all $\m{t}\leq r_0$ and all $\m{x_d}\leq r_d$.

Since $r_d<r'_d$, these inequalities prove in particular that the formal power series $\partial_{m_0;t}^{i_0}\partial_{m_1;x_1}^{i_1}...\partial_{m_n;x_n}^{i_n}a(t,x)$ is normally convergent on the closed polydisc $\overline{D}_{r_0}\times\overline{D}_{r_1}\times...\times\overline{D}_{r_n}$; hence, on all the compact sets of $\DCn$. Consequently, it defines an analytic function on $\DCn$, which completes the proof.
\end{proof}
 
\section{Gevrey formal power series}\label{section:gevrey-series}

Let us now recall the definition and basic properties of formal power series of a given Gevrey order $\sigma$.

\begin{definition}[Gevrey order]\label{defi:gevrey-order}
 Let $\sigma\ge 0$. Then, a formal power series $$\widetilde{u}(t,x)=\dsum{j\ge 0}{}{\TC{u}{j}(x)t^j}\in\Domn$$
 is said to be \textit{Gevrey of order $\sigma$} (or, for short, \textit{$\sigma$-Gevrey}) if there exist a radius $0<r<\min\{\rho_1,\ldots,\rho_N\}$ and two positive constants $C,K>0$ such that the inequalities
 $$
\m{\TC{u}{j}(x)}\le CK^j \Gamma(1+\sigma j)
 $$
hold for all $x\in D_{r,...,r}$ and all $j\geq0$.
\end{definition}

In other words, Definition \ref{defi:gevrey-order} means that $\widetilde{u}(t,x)$ is $\sigma$-Gevrey in $t$, uniformly in $x$ on a neighborhood of $x=(0,...,0)\in\mathbb C^N$.

\begin{notation}
\textnormal{We denote by $\DomGn{\sigma}$ the set of all the formal series in $\Domn$ which are $\sigma$-Gevrey.}
\end{notation}

Observe that any formal power series in $\DomGn{0}$ defines an analytic function at the origin of $\C^{N+1}$.

Observe also that the sets $\DomGn{\sigma}$ are filtered as follows:
\begin{multline}
\DomGn{0}\subset\DomGn{\sigma}\\\subset\DomGn{\sigma'}\subset\Domn
\label{FiltrationGevreySpaces}
\end{multline}
for all $\sigma$ and $\sigma'$ satisfying $0<\sigma<\sigma'<+\infty$.

The proposition below specifies the algebraic structure of $\DomGn{\sigma}$.

\begin{proposition}[\cite{R23b}]\label{GevreyAlgStructure}
Let $\sigma\geq0$. Then, the set $\DomGn{\sigma}$ endowed with the usual algebraic operations and the usual derivations $\partial_t$ and $\partial_{x_{d}}$ with $d=1,...,N$ is a $\C$-differential algebra.
\end{proposition}

With respect to moment derivations $\partial_{m_0;t}$ and $\partial_{m_d;x_d}$, we can also prove the following.

\begin{proposition}\label{GevreyMomentDerivatives}
Let $m_0,m_1,...,m_N$ be $N+1$ moment functions and $\widetilde{u}(t,x)\in\Domn$ a $\sigma$-Gevrey formal power series with $\sigma\geq0$. Then,
\begin{enumerate}
\item the formal power series $\partial_{m_0;t}^{i_0}\widetilde{u}(t,x)$ is still $\sigma$-Gevrey for any $i_0\geq0$.
\item  the formal power series $\partial_{m_1;x_1}^{i_1}...\partial_{m_N;x_N}^{i_N}\widetilde{u}(t,x)$ is still $\sigma$-Gevrey for any $i_1,....,i_N\geq0$.
\end{enumerate}
\end{proposition}

\begin{proof}
The proof of the first point is similar to the one of Proposition \ref{AnalyticFunctionMomentDerivatives} and is left to the reader. As for the proof of the second point, it is much more complicated and is essentially based on the integral representation of moment derivatives of analytic functions at the origin of $\C^N$ (\cite[Prop. 3]{Mic13c}). We refer to \cite{Rsub} for more details.
\end{proof}

Observe that Proposition \ref{GevreyMomentDerivatives} does not say that the set $\DomGn{\sigma}$ is stable under the moment derivatives $\partial_{m_d;x_d}$, since we have a priori no control on the domain of analyticity of the function $\partial_{m_1;x_1}^{i_1}...\partial_{m_N;x_N}^{i_N}\TC{u}{j}(x)$. However, when we consider only regular moment functions (see Proposition \ref{AnalyticFunctionMomentDerivatives}), we can state the following.

\begin{corollary}\label{StabilityGevreyMomentDerivative}
Let $m_0,m_1,...,m_N$ be $N+1$ regular moment functions and $\sigma\geq0$. Then, the set $\DomGn{\sigma}$ is stable under the moment derivatives $\partial_{m_0;t}$ and $\partial_{m_d;x_d}$ for all $d=1,...,N$.
\end{corollary}

\section{Modified Nagumo norms}\label{section:nagumo}

In this section we introduce the concept of modified Nagumo norms, aiming to create a tool with properties similar to standard Nagumo norms (see \cite{Nagumo,C-DRSS}) that can be used in the framework of moment differential operators. Below we expand on the idea introduced in \cite{Su21}.

\begin{notation}\label{Not}
For any $\alpha\geq0$ and $s>0$, we consider the formal power series
$$\Theta_{\alpha,s}(x)=\dsum{j\geq0}{}{\dbinom{\alpha+j-1}{j}^s x^j}$$
with
$$\dbinom{\alpha+j-1}{j}=\dfrac{\Gamma(\alpha+j)}{\Gamma(1+j)\Gamma(\alpha)}=\begin{cases}1&\text{if }j=0\\
\dfrac{\alpha(\alpha+1)...(\alpha+j-1)}{j!}&\text{if }j\geq1\end{cases}.$$
In particular,
\begin{align*}
&\Theta_{0,s}(x)=1\quad\text{and}\\
&\Theta_{\alpha,1}(x)=1+\dsum{j\geq1}{}{\dfrac{\alpha(\alpha+1)...(\alpha+j-1)}{j!}x^j}=\dfrac{1}{(1-x)^\alpha}\ \textrm{ for }|x|<1. 
\end{align*}
\end{notation}

The definition of the modified Nagumo norms is based on the classical notion of \textit{majorant series}. Recall that a formal power series $$\widetilde{V}(x)=\dsum{j_1,...,j_N\geq0}{}{V_{j_1,...,j_N}x_1^{j_1}...x_N^{j_N}}\in\R^+[[x]]$$ is called \textit{a majorant series} of
$$\widetilde{v}(x)=\dsum{j_1,...,j_N\geq0}{}{v_{j_1,...,j_N}x_1^{j_1}...x_N^{j_N}}\in\C[[x]]$$
if $\m{v_{j_1,...,j_N}}\leq V_{j_1,...,j_N}$ for all $j_1,...,j_N\geq0$. In this case, we denote $\widetilde{v}(x)\ll\widetilde{V}(x)$.

\begin{definition}[Modified Nagumo norms]\label{ModifiedNagumoNorm}
Let $f(x)=\dsum{j_1,...,j_N\geq 0}{}{f_{j_1,...,j_N}x_1^{j_1}...x_N^{j_N}}\in\domn$ be an analytic function on a polydisc $\DCn$. Moreover, let $s=(s_1,...,s_N)\in(\R^*_+)^N$ and $\alpha=(\alpha_1,...,\alpha_N)\in[1,+\infty[^N\cup\{0\}$ and suppose that $0<r<\min(\rho_1,...,\rho_N)$. Then, the modified Nagumo norm $\MNN{f}{\alpha}{r}{s}$ of $f$ with indices $(\alpha,r,s)$ is defined by:
$$\MNN{f}{\alpha}{r}{s}=\begin{cases}
\dsum{j_1,...,j_N\geq0}{}{\m{f_{j_1,...,j_N}}r^{j_1+...+j_N}}&\text{if }\alpha=0\\
\inf\left(A\geq0:f(x)\ll A\dprod{d=1}{N}{\dfrac{1}{r^{\alpha_d}}\Theta_{\alpha_d, s_d}\left(\dfrac{x_d}{r}\right)}\right)&\text{otherwise}
\end{cases}.$$\label{defi:ModNagumoNorm}
\end{definition}

\begin{remark} The modified Nagumo norms are well defined for $\alpha\in[1,+\infty[^N$. To prove that this is indeed the case, firstly let us notice that every $f(x)\in\domn$ has a majorant series of the form $\dsum{j_1,...,j_N\geq 0}{}{|f_{j_1,...,j_N}|x_1^{j_1}...x_N^{j_N}}$. From this it follows that
\begin{multline*}
f(x)\ll r^{\lambda(\alpha)}\dsum{j_1,...,j_N\geq 0}{}{|f_{j_1,...,j_N}|\dfrac{x_1^{j_1}}{r^{\alpha_1}}...\dfrac{x_N^{j_N}}{r^{\alpha_N}}}\\\ll r^{\lambda(\alpha)}\sum_{j\ge 0}\dsum{\begin{subarray}{c}j_1,...,j_N\geq 0\\j_1+...+j_N=j\end{subarray}}{}{|f_{j_1,...,j_N}|\dfrac{x_1^{j_1}}{r^{\alpha_1}}...\dfrac{x_N^{j_N}}{r^{\alpha_N}}}
\end{multline*}
Furthermore, let us notice that 
\begin{equation}
\dbinom{\alpha_d+j-1}{j}^{s_d}\ge 1\ \textrm{for any}\ d=1,\ldots,N.\label{eqn:binomial}
\end{equation}
Moreover, seeing as $f(x)$ is an analytic function, every coefficient 
$|f_{j_1,...,j_N}|$ can be bounded from above by $Mr^{-(j_1+\ldots+j_N)}$ with $M=\sup_{|\xi|\le r}|f(\xi)|$. The conclusion follows directly from these facts.

Furthermore, let us observe that for $\alpha\in]0,1[^N$, inequality from (\ref{eqn:binomial}) fails. Indeed, the coefficient  $\dbinom{\alpha_d+j-1}{j}$ decreases when $n$ tends to infinity. Hence, for any fixed $\alpha_d\in]0,1[$ we have
$$\lim_{j\to+\infty}\dbinom{\alpha_d+j-1}{j}=0$$
from the Stirling's Formula. Consequently, the modified Nagumo norms cannot be defined  this way when $\alpha_d\in]0,1[$.
\end{remark}

\begin{proposition}\label{norm}
For fixed ($\alpha,r,s)$, the function $\MNN{f}{\alpha}{r}{s}:\domn\to\R^+$ defines a norm on $\domn$.
\end{proposition}

\begin{proof} Let us fix $\alpha$, $r$ and $s$.
For any function $f$ obviously $\MNN{f}{\alpha}{r}{s}\ge 0$ and equality holds only for $f\equiv 0$. Moreover, for any constant $C$ equality $\MNN{Cf}{\alpha}{r}{s}=|C|\MNN{f}{\alpha}{r}{s}$ holds following from the definition of the majorant series. As such it remains to show that the triangle inequality holds for any two functions $f,g\in\domn$. To this end, notice that for functions $f(x)=\dsum{j_1,...,j_N\geq 0}{}{f_{j_1,...,j_N}x_1^{j_1}...x_N^{j_N}}$ and $g(x)=\dsum{j_1,...,j_N\geq0}{}{g_{j_1,...,j_N}x_1^{j_1}...x_N^{j_N}}$ we have $|f_{j_1,...,j_N}+g_{j_1,...,h_N}|\le|f_{j_1,...,j_N}|+|g_{j_1,...,j_N}|$ for every $j_1,...,j_N\geq0$. The conclusion follows directly from the definition of the majorant series.
\end{proof}

For the remainder of this section, we shall show that the modified Nagumo norms given in Definition \ref{defi:ModNagumoNorm} have properties similar to the classical Nagumo norms. To this end, results from \cite{Su21} will be adapted to the slightly more general case considered in this paper.

Let us start with two elementary technical lemmas.

\begin{lemma}
Let $\alpha,\beta\geq0$. Then, the following identity holds for all integer $j\geq0$:
$$\dsum{k=0}{j}{\dbinom{\alpha+k-1}{k}\dbinom{\beta+j-k-1}{n-k}}=\dbinom{\alpha+\beta+j-1}{j}.$$
\end{lemma}

\begin{proof}
It is sufficient to observe that the identity
$$\dfrac{1}{(1-x)^{\alpha+\beta}}=\dfrac{1}{(1-x)^\alpha}\times\dfrac{1}{(1-x)^\beta}$$
implies
\begin{align*}
\dsum{j\geq0}{}{\dbinom{\alpha+\beta+j-1}{j}x^j}&=\left(\dsum{j\geq0}{}{\dbinom{\alpha+j-1}{j}x^j}\right)\left(\dsum{j\geq0}{}{\dbinom{\beta+j-1}{j}x^j}\right)\\
&=\dsum{j\geq0}{}{\left(\dsum{k=0}{j}{\dbinom{\alpha+k-1}{k}\dbinom{\beta+j-k-1}{j-k}}\right)x^j}.
\end{align*}
\end{proof}

\begin{lemma}\label{Lemma}
Let $0\leq a\leq b$ and $\alpha\geq0$. Then, $\dbinom{a+\alpha}{a}\leq\dbinom{b+\alpha}{b}$.
\end{lemma}

\begin{proof}
The inequality is clear for $\alpha=0$ and for $a=b=0$. Let us now fix $\alpha,b>0$ and let us consider the function $f_b:a\in[0,b]\longmapsto\dbinom{a+\alpha}{a}$. Its derivative is given by
$$f'_b(a)=\dbinom{a+\alpha}{a}(\Psi(1+a+\alpha)-\Psi(1+a))$$
with is positive since $\Psi$ is an increasing function on $]0,+\infty[$ (the function $\ln\Gamma$ is convex on $]0,+\infty[$). Lemma \ref{Lemma} follows.
\end{proof}

In the first two results below, we are interested in the modified Nagumo norms of a product.

\begin{proposition}[Adaptation of \cite{Su21}, Lemma 2]\label{NagumoNormProduct}
Let $f(x),g(x)\in\domn$.\\
Let $s\in[1,+\infty[^N$.
Let $\alpha,\beta\in[1,+\infty[^N\cup\{0\}$ and $0<r<\min(\rho_1,...,\rho_N)$.\\
Then, $\MNN{fg}{\alpha+\beta}{r}{s}\leq\MNN{f}{\alpha}{r}{s}\MNN{g}{\beta}{r}{s}$.
\end{proposition}

\begin{proof}
First let us consider $(\alpha,\beta)\ne (0,0)$. Then
\begin{equation*}
f(x)\ll\frac{\MNN{f}{r}{s}{\alpha}}{r^{\lambda(\alpha)}}\prod_{d=1}^{N}\sum_{j\geq0}\binom{\alpha_{d}+j-1}{j}^{s_{d}}\frac{x_{d}^{j}}{r^{j}}
\end{equation*}
and
\begin{equation*}
g(x)\ll\frac{\MNN{g}{r}{s}{\beta}}{r^{\lambda(\beta)}}\prod_{d=1}^{N}\sum_{j\geq0}\binom{\beta_{d}+j-1}{j}^{s_{d}}\frac{x_{d}^{j}}{r^{j}}.
\end{equation*}
From this and the definition of the majorant series, it follows that 
\begin{equation*}
 f(x)g(x)\ll\dfrac{\MNN{f}{r}{s}{\alpha}\MNN{g}{r}{s}{\beta}}{r^{\lambda(\alpha)+\lambda(\beta)}}
 \prod_{d=1}^{N}\sum_{j\geq0}\left(\sum_{k=0}^{j}\binom{\alpha_{d}+k-1}{k}^{s_{d}}\binom{\beta_{d}+j-k-1}{j-k}^{s_{d}}\right)\frac{x_{d}^{j}}{r^{j}}.
\end{equation*}
Let us now notice that $a^s+b^s\le (a+b)^s$ for any $a,b>0$ and $s\ge 1$. Hence, for any $j\ge 0$ and $d=1,\ldots N$, we get
$$
\sum_{k=0}^{j}\binom{\alpha_{d}+k-1}{k}^{s_{d}}\binom{\beta_{d}+j-k-1}{j-k}^{s_{d}}\le\left(\sum_{k=0}^{j}\binom{\alpha_{d}+k-1}{k}\binom{\beta_{d}+j-k-1}{j-k}\right)^{s_d}.
$$
Using this fact and Lemma \ref{Lemma}, we conclude that
\begin{equation*}
 f(x)g(x)\ll\dfrac{\MNN{f}{r}{s}{\alpha}\MNN{g}{r}{s}{\beta}}{r^{\lambda(\alpha+\beta)}}
 \prod_{d=1}^{N}\sum_{j\geq0}\binom{\alpha_{d}+\beta_d+j-1}{j}^{s_{d}}\frac{x_{d}^{j}}{r^{j}};
\end{equation*}
hence, $\MNN{fg}{\alpha+\beta}{r}{s}\leq\MNN{f}{\alpha}{r}{s}\MNN{g}{\beta}{r}{s}$ thanks to Definition \ref{ModifiedNagumoNorm}.

Now, let us suppose that $\alpha=0$ and $\beta\in[1,+\infty[^N$. Then
\begin{multline*}
 f(x)g(x)\ll\left(\dsum{j_1,...,j_N\geq0}{}{\m{f_{j_1,...,j_N}}r^{j_1+...+j_N}\prod_{d=1}^N\frac{x^{i_d}}{r^{i_d}}}\right)\times\\\left(\dfrac{\MNN{g}{r}{s}{\beta}}{r^{\lambda(\beta)}}
 \prod_{d=1}^{N}\sum_{j\geq0}\binom{\beta_{d}+j-1}{j}^{s_{d}}\frac{x_{d}^{j}}{r^{j}}\right)
\end{multline*}
and the conclusion follows from Lemma \ref{Lemma}. Of course, the same holds when $\alpha\in[1,+\infty[^N$ and $\beta=0$.

For $\alpha=\beta=0$, the inequality is obviously true.
\end{proof}

\begin{remark}
 Note that the assumption \textquotedblleft$s_d\ge 1$ for all $d=1,\ldots,N$\textquotedblright\ is necessary for the proposition above to hold true. Otherwise, it is not possible to use inequalities of the form $a^s+b^s\le (a+b)^s$.
\end{remark}

Considering in particular the case $g(x)=1$, we can easily derive from Proposition \ref{NagumoNormProduct} the following. 

\begin{corollary}\label{NagumoNormProductWithOne}
 For all $\alpha\in[1,+\infty[^N\cup\{0\}$ and $\beta\in[1,+\infty[^N$, we have $\MNN{f}{\alpha+\beta}{r}{s}\le r^{\lambda(\beta)}\MNN{f}{\alpha}{r}{s}$.
\end{corollary}

The following two results show the action of the moment derivatives on the modified Nagumo norms.

\begin{proposition}[\cite{Su21}, Lemma 4]
Let $f(x)\in\domn$.\\
Let $s\in(\R^*_+)^N$, $\alpha\in[1,+\infty[^N$ and $0<r<\min(\rho_1,...,\rho_N)$.\\
Let $e_d\in(\R^+)^N$ be the multi-index with a $1$ in the $d$-th coordinate and zeros everywhere else.\\
Let $m_d$ be a regular moment function of order $s_d$. Then, there exists a positive constant $A>0$ such that
$$\MNN{\partial_{m_d;x_d}f}{\alpha+e_d}{r}{s}\leq C\alpha_d^{s_d}\MNN{f}{\alpha}{r}{s}$$
\end{proposition}

\begin{proof}
Since $m_d$ is a regular moment function, there exist positve constants $a,A>0$ such that
 $$
 a(j+1)^s\le\frac{m_d(j+1)}{m_d(j)}\le A(j+1)^s\quad\textrm{for every}\quad j\in\Na.
 $$
 Moreover, let us notice that 
 \begin{align*}
 \partial_{m_d;x_d}f(x)&\ll\dfrac{\MNN{f}{\alpha}{r}{s}}{r^{\lambda(\alpha)}}{\dsum{j\geq 0}{}\dbinom{\alpha_d+j}{j+1}^{s_d} \frac{m_d(j+1)}{m_d(j)}\frac{x_d^j}{r^{j+1}}\dprod{i\ne d}{}{\dbinom{\alpha_i+j-1}{j}^{s_i} \frac{x_i^j}{r^j}}}\\
 &\ll\dfrac{A\MNN{f}{\alpha}{r}{s}}{r^{\lambda(\alpha)+1}}{\dsum{j\geq 0}{}\dbinom{\alpha_d+j}{j+1}^{s_d} (j+1)^{s_d}\frac{x_d^j}{r^{j}}\dprod{i\ne d}{}{\dbinom{\alpha_i+j-1}{j}^{s_i} \frac{x_i^j}{r^j}}}.
\end{align*}
Since
\begin{align*}
\dbinom{\alpha_d+j}{j+1}&=\frac{\Gamma(\alpha_d+j+1)}{\Gamma(2+j)\Gamma(\alpha_d)}\\
&=\frac{\alpha_d\Gamma(\alpha_d+j+1)}{(j+1)\Gamma(1+j)\Gamma(\alpha_d+1)}=\frac{\alpha_d}{j+1}\dbinom{\alpha_d+1+j-1}{j},
\end{align*}
we receive
$$
\partial_{m_d;x_d}f(x)\ll\dfrac{A\alpha_d^{s_d}\MNN{f}{\alpha}{r}{s}}{r^{\lambda(\alpha)+1}}{\dsum{j\geq 0}{}\dbinom{\alpha_d+1+j-1}{j}^{s_d}\frac{x_d^j}{r^{j}}\dprod{i\ne d}{}{\dbinom{\alpha_i+j-1}{j}^{s_i} \frac{x_i^j}{r^j}}}.
$$
and the conclusion follows.
\end{proof}

\begin{corollary}\label{NagumoNormDerivative}
Assume that $m_1,...,m_N$ are all regular moment functions. Then, for all $\alpha\in[1,+\infty[^N$ and all $q\in\Na^N$, there exists a positive constant $C>0$ such that
$$\MNN{\partial_{m;x}^qf}{\alpha+q}{r}{s}\leq C^{\lambda(q)}\left(\dprod{d=1}{N}{q_d!^{s_d}\dbinom{\alpha_d+q_d-1}{q_d}^{s_d}}\right)\MNN{f}{\alpha}{r}{s}.$$
\end{corollary}

Note that if $q_d=0$, then the corresponding term in the product is $1$ (see Notation \ref{Not}). In particular, this inequality remains valid when $q=0$.

The last two properties will enable us to link the modified Nagumo norms with the concept of Gevrey order of a formal power series.

\begin{proposition}\label{prop:gevrey_norm}
Let $$\widetilde{u}(t,x)=\dsum{j\geq0}{}{\TC{u}{j}(x)t^j}\in\DomGn{\sigma}$$ a $\sigma$-Gevrey formal power series for a certain
$\sigma\geq 0$. Let $0<r<\min\{\rho_1,\ldots,\rho_N\}$ as in Definiton \ref{defi:gevrey-order}. Then, for all $\alpha\in[1,+\infty[^{N}\cup\{0\}$ and all $s\in(\R^*_+)^N$, there exist two positive constants $A,B>0$ such that the following inequality holds for all $j\geq0$:
\begin{equation*}
\MNN{\TC{u}{j}}{j\alpha}{r}{s}\leq AB^{j}\Gamma(1+\sigma j).
\end{equation*}
\end{proposition}

\begin{proof}
 The proof is identical to the one presented in \cite{Su21} and follows directly from Definition \ref{defi:gevrey-order} and the Cauchy formula.
\end{proof}

\begin{proposition}\label{prop:sup_norm}
Let $0<\rho<r<\min(\rho_1,\ldots,\rho_N)$. Then, there exists a positive constant $A>0$ such that, for all $f(x)\in\domn$ and all $\alpha\in[1,+\infty[^{N}\cup\left\{0\right\}$, the following inequality holds for all $x\in D_{\rho,...,\rho}$:
\begin{equation*}
\m{f(x)}\leq A^{\lambda(\alpha)}\MNN{f}{\alpha}{r}{s}.
\end{equation*}
\end{proposition}

\begin{proof}
 For $\alpha=0$, the inequality is obviously true. To show the same for any $\alpha\in[1,+\infty[^N$, let us first notice that for any $a,b\in\R^+$, $p\in[1,+\infty[$ and $j\in\Na$ inequality
 $$
 \dbinom{j+p-1}{j}a^jb^{p-1}\le(a+b)^{j+p-1}
 $$
 holds. Then, if we take $a+b=1$ with $a^{-1}=1+\varepsilon$ for any $\varepsilon>0$, we receive
 $$
 \dbinom{j+p-1}{j}(1+\varepsilon)^{-j}\left(1-\dfrac{1}{1+\varepsilon}\right)^{p-1}\le 1,
 $$
 and then
 \begin{equation}
 \dbinom{j+p-1}{j}\le (1+\varepsilon)^{j}\left(\frac{1+\varepsilon}{\varepsilon}\right)^{p-1}.\label{eqn:4}
 \end{equation}
 
 Let us then fix $\varepsilon>0$ sufficiently small that $\rho(1+\varepsilon)^{\lambda(s)-N}<r$ holds true. We can use inequality (\ref{eqn:4}) to find a majorant series for $f(x)$. More precisely, we get
 \begin{align*}
  f(x)&\ll\frac{\MNN{f}{\alpha}{r}{s}}{r^{\lambda(\alpha)}}\dprod{d=1}{N}{\sum_{j\ge 0}\dbinom{\alpha_d+j-1}{j}^{s_d}\frac{x_d^j}{r^j}}\\
  &\ll\frac{\MNN{f}{\alpha}{r}{s}}{r^{\lambda(\alpha)}}\left(\frac{1+\varepsilon}{\varepsilon}\right)^{\sum_{d=1}^N\alpha_d s_d}\dprod{d=1}{N}{\sum_{j\ge 0}\dbinom{\alpha_d+j-1}{j}\left(\frac{x_d (1+\varepsilon)^{s_d-1}}{r}\right)^j}
 \end{align*}
 Considering our previous restriction on $\varepsilon$ as well as Notation \ref{Not}, we can notice that 
 \begin{align*}
 \sup_{x\in D_{\rho,...,\rho}}\left|f(x)\right|&\le \frac{\MNN{f}{\alpha}{r}{s}}{r^{\lambda(\alpha)}}\left(\frac{1+\varepsilon}{\varepsilon}\right)^{\sum_{d=1}^N\alpha_d s_d}\frac{1}{\left(1-\frac{\rho(1+\varepsilon)^{\lambda(s)-N}}{r}\right)^{\lambda(\alpha)}}\\
& \le\frac{\MNN{f}{\alpha}{r}{s}}{r^{\lambda(\alpha)}}\left(\frac{1+\varepsilon}{\varepsilon}\right)^{\hat{s}\lambda(\alpha)}\frac{1}{\left(1-\frac{\rho(1+\varepsilon)^{\lambda(s)-N}}{r}\right)^{\lambda(\alpha)}}
 \end{align*}
 for $\hat{s}=\max(s_1,\ldots,s_N)$. This concludes the proof.
\end{proof}

We are now able to turn to our initial problem.

\section{Main results}\label{section:main}

Let us consider $N+1$ regular moment functions $m_0,m_1,\ldots,m_N$ of respective orders $s_0>0$ and $s_1,\ldots,s_N\geq 1$. In this section, we focus on the inhomogeneous nonlinear moment partial differential equations of the form
\begin{equation}\label{EQ1}
\begin{cases}
\partial_{m_0;t}^\kappa u - P(t,x,(\partial_{m_0;t}^i\partial_{m;x}^qu)_{(i,q)\in\Lambda})=\widetilde{f}(t,x)\\
\partial_{m_0;t}^ju(t,x)_{|t=0}=\varphi_j(x)\in\Oo(\DCn) \textrm{ for }0\le j<\kappa,
\end{cases}
\end{equation}
where the following conditions are met:
\begin{itemize}
\item$\kappa\geq1$ is a positive integer;
\item$\Lambda$ is a non-empty finite subset of $\{0,...,\kappa-1\}\times\Na^N$;
\item$\partial_{m;x}^q$ stands for the moment derivation $\partial_{m_1;x_1}^{q_1}...\partial_{m_N;x_N}^{q_N}$ while $q=(q_1,...,q_N)$;
\item$P$ is a polynomial with analytic coefficients on the polydisc $\DCN$;
\item$\widetilde{f}(t,x)\in\Domn$.
\end{itemize}
More precisely, we shall always assume that the polynomial $P$ reads in the form
\begin{multline}
P(t,x,(\partial_{m_0;t}^i\partial_{m;x}^qu)_{(i,q)\in\Lambda})=\\
\dsum{n\in\mathcal I}{}{\dsum{(\underline{i},\underline{q},\underline{r})\in\Lambda_n}{}{t^{v_{\underline{i},\underline{q},\underline{r}}}a_{\underline{i},\underline{q},\underline{r}}(t,x)\left(\partial_{m_0;t}^{i_1}\partial_{m;x}^{q_1}u\right)^{r_1}...\left(\partial_{m_0;t}^{i_n}\partial_{m;x}^{q_n}u\right)^{r_n}}},\label{eqn:operator_P}
\end{multline}
where:
\begin{itemize}
\item$\mathcal I$ is a non-empty finite subset of $\Na^*$;
\item for any $n\in\mathcal{I}$, the set $\Lambda_n$ is a non-empty finite subset of $n$-tuples
$$(\underline{i},\underline{q},\underline{r})=((i_1,q_1,r_1),...,(i_n,q_n,r_n))$$
composed of elements of $\{0,...,\kappa-1\}\times\Na^N\times\Na^*$, whose the pairs $(i_k,q_k)$ are all two by two distincts;
\item$v_{\underline{i},\underline{q},\underline{r}}$ is a nonnegative integer for every $(\underline{i},\underline{q},\underline{r})\in\Lambda_n$;
\item$a_{\underline{i},\underline{q},\underline{r}}(t,x)\in\mathcal O(\DCN)$ and $a_{\underline{i},\underline{q},\underline{r}}(0,x)\not\equiv0$ for every $(\underline{i},\underline{q},\underline{r})\in\Lambda_n$.
\end{itemize}

\begin{proposition}
Eq. (\ref{EQ1}) is formally well-posed.
\end{proposition}
\begin{proof}
Let us take the coefficients $a_{\underline{i},\underline{q},\underline{r}}(t,x)$ in the form
$$a_{\underline{i},\underline{q},\underline{r}}(t,x)=\dsum{j\geq0}{}{a_{\underline{i},\underline{q},\underline{r};j,*}(x)\dfrac{t^j}{m_0(j)}}$$
and the inhomogeneity $\widetilde{f}(t,x)$ in the form
$$\widetilde{f}(t,x)=\dsum{j\geq0}{}{\TC{f}{j}(x)\dfrac{t^j}{m_0(j)}}.$$
 The coefficients $\TC{u}{j}(x)$ of the formal solution $\widetilde{u}(t,x)$ of Eq. (\ref{EQ1}) given in a similar form are uniquely determined by the recursion formul\ae
\begin{equation}\label{uj}\TC{u}{j+\kappa}(x)=\TC{f}{j}(x)+\dsum{n\in\mathcal I}{}{\dsum{(\underline{i},\underline{q},\underline{r})\in\Lambda_n}{}{\dsum{{j_0+j_1+...+j_{r_1+...+r_n}=j-v_{\underline{i},\underline{q},\underline{r}}}}{}{C_{\underline{i},\underline{q},\underline{r},\underline{j},n}(x)}}}\end{equation}
together with the initial conditions $\TC{u}{j}(x)=\varphi_j(x)$ for $j=0,...,\kappa-1$, where
\begin{multline}\label{Cjn}
C_{\underline{i},\underline{q},\underline{r},\underline{j},n}(x)=\dbinom{j}{j_0,...,j_{r_1+...+r_n}}_{m_0}\TC{a}{\underline{i},\underline{q},\underline{r};j_0}(x)\times\\
\dprod{\ell=1}{n}{\dprod{h=j_{r_1+...+r_{\ell-1}+1}}{j_{r_1+...+r_\ell}}{\partial_{m;x}^{q_\ell}\TC{u}{h+i_\ell}(x)}}.
\end{multline}

The notation $\dbinom{j}{j_0,...,j_{r_1+...+r_n}}_{m_0}$ stands for the moment multinomial coefficient of the form
$$\dbinom{j}{j_0,...,j_{r_1+...+r_n}}_{m_0}=\dfrac{m_0(j)}{m_0(j_0)m_0(j_1)...m_0(j_{r_1+...+r_n})}.$$
As usual, the third sum in (\ref{uj}) is zero as soon as $j<v_{\underline{i},\underline{q},\underline{r}}$, and the term $r_1+...+r_{\ell-1}$ in (\ref{Cjn}) is $0$ when $\ell=1$ so that $j_{r_1+...+r_{\ell-1}+1}=j_1$.

Observe that the fact that all the coefficients $\TC{u}{j}(x)$ are analytic on $\DCn$ is guaranteed by the assumption \textquotedblleft $m_1,...,m_N$ are regular moment functions\textquotedblright\ and Proposition \ref{AnalyticFunctionMomentDerivatives}.
\end{proof}

Let us now denote by $C(a,b)=\{(x,y)\in\R^2;x\leq a\text{ and }y\geq b\}$ for all $a,b\in\R$. Drawing inspiration from \cite{Yonemura} as well as various papers concerning moment differential equations (see for example \cite{MichSu20,Su21}), we define the Newton polygon for the nonlinear operator $\Delta_{\kappa,P}:=\partial_{m_0;t}^\kappa-P(t,x,(\partial_{m_0;t}^i\partial_{m;x}^q)_{(i,q)\in\Lambda})$ associated with Eq. (\ref{EQ1}) as follows.

\begin{definition}\label{defi:NewtonPolygon}
We call \textit{moment Newton polygon of $\Delta_{\kappa,P}$, and we denote it by $\mathcal N(\Delta_{\kappa,P})$,} the convex hull of
$$C(s_0\kappa,-\kappa)\cup\bigcup_{n\in\mathcal I}\bigcup_{(\underline{i},\underline{q},\underline{r})\in\Lambda_n}C\left(\dsum{\ell=1}{n}{\left(s_0r_\ell i_\ell+r_\ell\lambda(sq_\ell)\right)},v_{\underline{i},\underline{q},\underline{r}}-\dsum{\ell=1}{n}{r_\ell i_\ell}\right)$$
with
$$\lambda(sq_\ell)=\dsum{d=1}{N}{s_d q_{\ell,d}}.$$
\end{definition}

Further ahead the following assumption will be used:
\begin{assumption}\label{ASS1}
For all $n\in\mathcal I$ and all $(\underline{i},\underline{q},\underline{r})\in\Lambda_n$ we assume that $$\dsum{\ell=1}{n}{r_\ell i_\ell}-v_{\underline{i},\underline{q},\underline{r}}<\kappa.$$
\end{assumption}

The geometric structure of $\mathcal N(\Delta_{\kappa,P})$ is specified in the following.

\begin{proposition}\label{GeometricStructureMomentNewtonPolygon}
For any $n\in\mathcal I$, let us denote by $\mathcal S_n$ the set of all the the tuples $(\underline{i},\underline{q},\underline{r})\in\Lambda_n$ such that
$$\dsum{\ell=1}{n}{\left(s_0r_\ell i_\ell+r_\ell\lambda(sq_\ell)\right)}>s_0\kappa.$$
Let $\mathcal S=\displaystyle\bigcup_{n\in\mathcal I}\mathcal S_n$.
\begin{enumerate}
\item Assume $\mathcal S=\emptyset$. Then, the moment Newton polygon $\mathcal N(\Delta_{\kappa,P})$ is reduced to the domain $C(s_0\kappa,-\kappa)$. In particular, it has no side with a positive slope (see Fig. \ref{MomentNewtonPolygonPropFig1}).
\item Assume $\mathcal S\neq\emptyset$. Then, the moment Newton polygon $\mathcal N(\Delta_{\kappa,P})$ has at least one side with a positive slope. Moreover, its smallest positive slope $k$ is given by
\begin{multline*}
k=\min_{\substack{n\in\mathcal I\\(\underline{i},\underline{q},\underline{r})\in\mathcal S_n}}\left(\dfrac{\kappa+v_{\underline{i},\underline{q},\underline{r}}-\dsum{\ell=1}{n}{r_\ell i_\ell}}{\dsum{\ell=1}{n}{\left(s_0r_\ell i_\ell+r_\ell\lambda(sq_\ell)\right)}-s_0\kappa}\right)\\=\frac{\kappa+v_{\underline{i}^*,\underline{q}^*,\underline{r}^*}-\dsum{\ell=1}{n^*}{r_\ell^* i_\ell^*}}{\dsum{\ell=1}{n^*}{\left(s_0r_\ell^* i_\ell^*+r_\ell^*\lambda(sq_\ell^*)\right)}-s_0\kappa},
\end{multline*}
where $n^*\in\mathcal I$ and the tuple $(\underline{i}^*,\underline{q}^*,\underline{r}^*)\in\mathcal S_{n*}$ are chosen (see Fig. \ref{MomentNewtonPolygonPropFig2}) in such a way that the edge with slope $k$ is the segment with end points $(s_0\kappa,-\kappa)$ and $$\left(\dsum{\ell=1}{n^*}{\left(s_0r_\ell^* i_\ell^*+r_\ell^*\lambda(sq_\ell^*)\right)},v_{\underline{i}^*,\underline{q}^*,\underline{r}^*}-\dsum{\ell=1}{n^*}{r_\ell^* i_\ell^*}\right).$$
\end{enumerate}
\begin{center}
\begin{figure}[h]
	\begin{center}
		\begin{subfigure}[b]{5cm}
		\centering\begin{tikzpicture}
		\draw[fill=gray!10,color=gray!10] (-1,-1) -- (1,-1) -- (1,2.5) -- (-1,2.5) -- cycle;
		\draw(-1,0) -- (2.5,0);
		\draw(0,-1.5) -- (0,2.5);
		\draw (0,-1) node{-} node[below left]{$-\kappa$};
		\draw (1,0) node[rotate=90]{-} node[below right]{$s_0\kappa$};
		\draw (0,1.75) node{$\bullet$}; \draw (1/3,-2/3) node{$\bullet$};
		\draw (1/3,0) node{$\bullet$}; \draw (1/3,0) node{$\bullet$};
		\draw (2/3,0.75) node{$\bullet$}; \draw (2/3,1.25) node{$\bullet$}; \draw (2/3,-1/3) node{$\bullet$};
		\draw[line width=1pt] (-1,-1) -- (1,-1) node{$\bullet$} -- (1,2.5); \draw (1,0.375) node{$\bullet$};
		\draw (0,0) node[below left]{$0$};
		\end{tikzpicture}
		\caption{Case $\mathcal S=\emptyset$}\label{MomentNewtonPolygonPropFig1}
		\end{subfigure}
		\begin{subfigure}[b]{6cm}
		\centering\begin{tikzpicture}[>=latex]
		\draw[fill=gray!10,color=gray!10] (-1,-1) -- (1,-1) -- (2,-0.5) -- (2.5,0.25) -- (2.9,1.3) -- (2.9,2.5) -- (-1,2.5) -- cycle;
		\draw(-1,0) -- (3.5,0);
		\draw(0,-1.5) -- (0,2.5);
		\draw (0,-1) node{-} node[below left]{$-\kappa$};
		\draw (1,0) node[rotate=90]{-} node[above]{$s_0\kappa$};
		\draw[dashed] (1,0) -- (1,-1);
		\draw [dashed] (2,0) node[rotate=90]{-} -- (2,-0.5);
		\draw[->] (2.5,-0.5) node[right]{\scriptsize$\dsum{\ell=1}{n^*}{\left(s_0r_\ell^* i_\ell^*+r_\ell^*\lambda(sq_\ell^*)\right)}$} -- (2,0);
		\draw [dashed] (0,-0.5) node{-} node[left]{\scriptsize{$v_{\underline{i}^*,\underline{q}^*,\underline{r}^*}-\dsum{\ell=1}{n^*}{r_\ell^* i_\ell^*}$}} -- (2,-0.5);
		\draw[line width=1pt] (-1,-1) -- (1,-1) node{$\bullet$} -- node[below,sloped]{\textnormal{\footnotesize{slope $k$}}} (2,-0.5) node{$\bullet$} -- (2.5,0.25) node{$\bullet$};
		\draw[line width=1pt,densely dashed] (2.5,0.25) -- (2.9,1.3) node{$\bullet$};
		\draw[line width=1pt] (2.9,1.3) -- (2.9,2.5);
		\draw (0,1.75) node{$\bullet$}; \draw (0.5,-0.5) node{$\bullet$}; \draw (0.5,0) node{$\bullet$};
		\draw (1,1.5) node{$\bullet$};
		\draw (1.5,1.5) node{$\bullet$}; \draw (1.5,0) node{$\bullet$}; \draw (1.5,-0.75) node{$\bullet$};
		\draw (1.5,0.5) node{$\bullet$}; \draw (2,2) node{$\bullet$}; \draw (2,1.25) node{$\bullet$}; \draw (2,0.75) node{$\bullet$};
		\draw (2.5,0.75) node{$\bullet$}; \draw (2.5,1.25) node{$\bullet$};
		\draw (2.9,1.7) node{$\bullet$};
		\draw (0,0) node[above left]{$0$};
		\end{tikzpicture}
		\caption{Case $\mathcal S\neq\emptyset$}\label{MomentNewtonPolygonPropFig2}
		\end{subfigure}
	\end{center}
	\caption{The moment Newton polygon $\mathcal N(\Delta_{\kappa,P})$ associated with Eq. (\ref{EQ1})}
\end{figure}
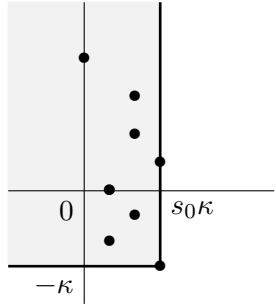
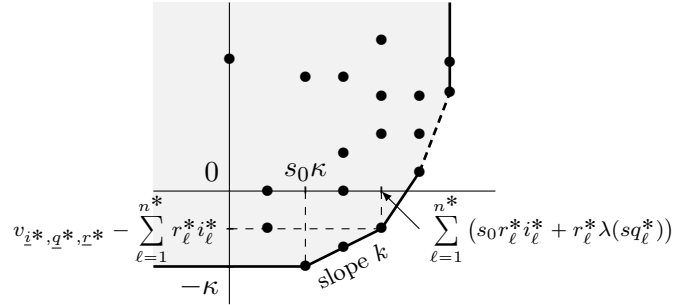
\end{center}
\end{proposition}

\begin{proof}
The first point stems obviously from the fact that the condition $\mathcal S=\emptyset$ implies
$$C\left(\dsum{\ell=1}{n}{\left(s_0r_\ell i_\ell+r_\ell\lambda(sq_\ell)\right)},v_{\underline{i},\underline{q},\underline{r}}-\dsum{\ell=1}{n}{r_\ell i_\ell}\right)\subset C(s_0\kappa,-\kappa)$$ for all $(\underline{i},\underline{q},\underline{r})\in\Lambda_n$ and all $n\in\mathcal I$. As for the second point, it suffices to remark, on one hand, that
$$C\left(\dsum{\ell=1}{n}{\left(s_0r_\ell i_\ell+r_\ell\lambda(sq_\ell)\right)},v_{\underline{i},\underline{q},\underline{r}}-\dsum{\ell=1}{n}{r_\ell i_\ell}\right)\subset C(s_0\kappa,-\kappa)$$ for all tuples $(\underline{i},\underline{q},\underline{r})\not\in\mathcal S$, and, on the other hand, that the segment with the two end points $(s_0\kappa,-\kappa)$ and
$$\left(\dsum{\ell=1}{n}{\left(s_0r_\ell i_\ell+r_\ell\lambda(sq_\ell)\right)},v_{\underline{i},\underline{q},\underline{r}}-\dsum{\ell=1}{n}{r_\ell i_\ell}\right)$$ has a positive slope equal to
$$\dfrac{\kappa+v_{\underline{i},\underline{q},\underline{r}}-\dsum{\ell=1}{n}{r_\ell i_\ell}}{\dsum{\ell=1}{n}{\left(s_0r_\ell i_\ell+r_\ell\lambda(sq_\ell)\right)}-s_0\kappa}$$
for all tuples $(\underline{i},\underline{q},\underline{r})\in\mathcal S$.
\end{proof}

\begin{definition}\label{DefSigmac}
We call \textit{critical value} of Eq. (\ref{EQ1}) the nonnegative real number $\sigma_c$ defined by
$$\sigma_c=\begin{cases}
0&\text{if }\mathcal S=\emptyset\\
\dfrac{1}{k}=\dfrac{\dsum{\ell=1}{n^*}{\left(s_0r_\ell^* i_\ell^*+r_\ell^*\lambda(sq_\ell^*)\right)}-s_0\kappa}{\kappa+v_{\underline{i}^*,\underline{q}^*,\underline{r}^*}-\dsum{\ell=1}{n^*}{r_\ell^* i_\ell^*}}&\text{if }\mathcal S\neq\emptyset
\end{cases}.$$
\end{definition}

According to the definition of $\sigma_c$, we derive in particular from Proposition \ref{GeometricStructureMomentNewtonPolygon} the following inequalities which will play a fundamental role in the proof of our main result.

\begin{proposition}\label{IneqMomentNewtonPolygon}
Let $\sigma\geq\sigma_c$. Then,
$$(\sigma+s_0)\left(\kappa+v_{\underline{i},\underline{q},\underline{r}}-\dsum{\ell=1}{n}{r_\ell i_\ell}\right)\geq s_0v_{\underline{i},\underline{q},\underline{r}}+\dsum{\ell=1}{n}{r_\ell\lambda(sq_\ell)}$$
for all $n\in\mathcal I$ and all $(\underline{i},\underline{q},\underline{r})\in\Lambda_n$.
\end{proposition}

\begin{proof}
The following is an adaptation of the proof of \cite[Lemma 3.9]{R23b}. First let us consider the case when $\mathcal{S}=\emptyset$ and, consequently, $\sigma_c=0$. Then from the definition of the set $\mathcal S$ it follows that
\begin{multline*}
(\sigma+s_0)\left(\kappa+v_{\underline{i},\underline{q},\underline{r}}-\dsum{\ell=1}{n}{r_\ell i_\ell}\right)\ge s_0\left(\kappa+v_{\underline{i},\underline{q},\underline{r}}-\dsum{\ell=1}{n}{r_\ell i_\ell}\right)\\\ge s_0v_{\underline{i},\underline{q},\underline{r}}+\dsum{\ell=1}{n}{r_\ell\lambda(sq_\ell)}.
\end{multline*}

Let us now consider the case when $\mathcal{S}\ne\emptyset$. Then $\sigma_c$ is defined as follows:
$$
\sigma_c=\dfrac{\dsum{\ell=1}{n^*}{\left(s_0r_\ell^* i_\ell^*+r_\ell^*\lambda(sq_\ell^*)\right)}-s_0\kappa}{\kappa+v_{\underline{i}^*,\underline{q}^*,\underline{r}^*}-\dsum{\ell=1}{n^*}{r_\ell^* i_\ell^*}}
$$
and for any $(\underline{i},\underline{q},\underline{r})\in \mathcal{S}$ the following inequalities hold
$$
\sigma\ge\sigma_c\ge\dfrac{\dsum{\ell=1}{n}{\left(s_0r_\ell i_\ell+r_\ell\lambda(sq_\ell)\right)}-s_0\kappa}{\kappa+v_{\underline{i},\underline{q},\underline{r}}-\dsum{\ell=1}{n}{r_\ell i_\ell}}>0.
$$
Moreover, if $(\underline{i},\underline{q},\underline{r})
\not\in \mathcal{S}$, then 
$$\dsum{\ell=1}{n}{\left(s_0r_\ell i_\ell+r_\ell\lambda(sq_\ell)\right)}\le s_0\kappa.$$
From this it follows that
$$
\sigma\ge\sigma_c>0\ge \dfrac{\dsum{\ell=1}{n}{\left(s_0r_\ell i_\ell+r_\ell\lambda(sq_\ell)\right)}-s_0\kappa}{\kappa+v_{\underline{i},\underline{q},\underline{r}}-\dsum{\ell=1}{n}{r_\ell i_\ell}}.
$$

Hence,
$$
(\sigma+s_0)\left(\kappa+v_{\underline{i},\underline{q},\underline{r}}-\dsum{\ell=1}{n}{r_\ell i_\ell}\right)\ge s_0v_{\underline{i},\underline{q},\underline{r}}+\dsum{\ell=1}{n}{r_\ell\lambda(sq_\ell)}.
$$
\end{proof}

We can now present the main result of the paper:
\begin{theorem}\label{GevreyRegularityTheorem}
Let $\sigma_c$ be the critical value of Eq. (\ref{EQ1}). Then,
\begin{enumerate}
\item$\widetilde{u}(t,x)$ and $\widetilde{f}(t,x)$ are simultaneously $\sigma$-Gevrey for any $\sigma\geq\sigma_c$;
\item$\widetilde{u}(t,x)$ is generically $\sigma_c$-Gevrey while $\widetilde{f}(t,x)$ is $\sigma$-Gevrey with $\sigma<\sigma_c$.
\end{enumerate}
\end{theorem}

Before we move on to the proof of both parts of Theorem \ref{GevreyRegularityTheorem}, let us first formulate a corollary that deals with convergent inhomogeneity $\widetilde{f}(t,x)$ as well as some additional examples.

\begin{corollary}
Assume that the inhomogeneity $\widetilde{f}(t,x)$ of Eq. (\ref{EQ1}) is convergent. Then, the formal solution $\widetilde{u}(t,x)$ is either convergent or $1/k$-Gevrey, where $k$ stands for the smallest positive slope of its associated moment Newton polygon. 
\end{corollary}

\begin{example}
Let us consider the \textit{semilinear regular moment heat equation}
\begin{equation}\label{momentHE}
\begin{cases}
\partial_{m_0;t} u-t^va(t,x)\Delta_{m;x}u+b(t,x)u^r=\widetilde{f}(t,x)\\
u(0,x)=\varphi(x)\in\Oo(D_{\rho_1,...,\rho_N})
\end{cases}
\end{equation}
where
\begin{itemize}
\item $\Delta_{m;x}=\partial_{m_1;x_1}^2+...+\partial_{m_N;x_N}^2$ is the moment Laplace operator;
\item the degree $r$ of the power-law nonlinearity is an integer at least $2$;
\item the valuation $v$ is a nonnegative integer;
\item the coefficients $a(t,x)$ and $b(t,x)$ are analytic on a polydisc $D_{\rho_0,\rho_1,...,\rho_N}$ and $a(0,x)\not\equiv0$;
\item $\widetilde{f}(t,x)\in\Domn$.
\end{itemize}
The moment Newton polygon associated with Eq. (\ref{momentHE}) is as shown on Fig \ref{MomentNewtonPolygonmomentHE} below. If any exists, we define $d^*$ by $d^*=\max\{d\in\{1,...,N\}:\ 2s_d>s_0\}$.
\begin{center}
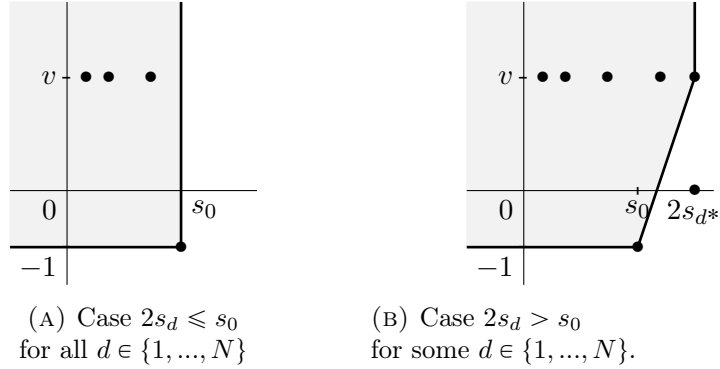
\begin{figure}[h]
	\begin{center}
		\begin{subfigure}[b]{6cm}
		\centering\begin{tikzpicture}
		\draw[fill=gray!10,color=gray!10] (-0.75,-0.75) -- (1.5,-0.75) -- (1.5,2.5) -- (-0.75,2.5) -- cycle;
		\draw(-0.75,0) -- (2.5,0);
		\draw(0,-1.25) -- (0,2.5);
		\draw (1.5,0) node[below right]{$s_0$};
		\draw (0,-0.75) node{-} node[below left]{$-1$};
		\draw[line width=1pt] (-0.75,-0.75) -- (1.5,-0.75) node{$\bullet$} -- (1.5,2.5);
		\draw (0,0) node[below left]{$0$};
		\draw (0,1.5) node{-} node[left]{$v$};
		\draw (0.25,1.5) node{$\bullet$}; \draw (0.55,1.5) node{$\bullet$}; \draw (1.1,1.5) node{$\bullet$};
		\end{tikzpicture}
		\caption{Case $2s_d\leq s_0$\\ for all $d\in\{1,...,N\}$}
		\end{subfigure}
		\begin{subfigure}[b]{6cm}
		\centering\begin{tikzpicture}[>=latex]
		\draw[fill=gray!10,color=gray!10] (-0.75,-0.75) -- (1.5,-0.75) -- (2.25,1.5) -- (2.25,2.5) -- (-0.75,2.5) -- cycle;
		\draw(-0.75,0) -- (2.5,0);
		\draw(0,-1.25) -- (0,2.5);
		\draw (1.5,0) node[rotate=90]{-}node[below]{$s_0$};
		\draw (2.25,0) node{$\bullet$} node[below]{$2s_{d^*}$};
		\draw (0,-0.75) node{-} node[below left]{$-1$};
		\draw[line width=1pt] (-0.75,-0.75) -- (1.5,-0.75) node{$\bullet$} -- (2.25,1.5) node{$\bullet$} -- (2.25,2.5);
		\draw (0,0) node[below left]{$0$};
		\draw (0,1.5) node{-} node[left]{$v$};
		\draw (0.25,1.5) node{$\bullet$}; \draw (0.55,1.5) node{$\bullet$}; \draw (1.1,1.5) node{$\bullet$}; \draw (1.8,1.5) node{$\bullet$};
		\end{tikzpicture}
		\caption{Case $2s_d>s_0$\\ for some $d\in\{1,...,N\}$.}
		\end{subfigure}
	\end{center}
	\caption{The moment Newton polygon associated with Eq. (\ref{momentHE})}\label{MomentNewtonPolygonmomentHE}
\end{figure}
\end{center}
The critical value of Eq. (\ref{momentHE}) is then defined by
$$\sigma_c=\begin{cases}
0&\text{if }2s_d\leq s_0\text{ for all }d\in\{1,...,N\}\\
\dfrac{2s_{d^*}-s_0}{1+v}&\text{otherwise}
\end{cases}$$
and the Gevrey regularity of the unique formal solution $\widetilde{u}(t,x)$ of Eq. (\ref{momentHE}) follows from Theorem \ref{GevreyRegularityTheorem}.
\end{example}
\begin{example}
Let us now consider the \textit{generalized regular moment Boussinesq equation}
\begin{equation}\label{momentBE}
\left\{
\begin{aligned}
&\partial_{m_0;t}^2 u-a(t,x)\partial_{m;x}^4u-P(t,x,u)\partial_{m;x}^2u-Q(t,x,u)(\partial_{m;x}u)^2=\widetilde{f}(t,x)\\
&\partial_{m_0;t}^ju(t,x)|_{t=0}=\varphi_j(x)\in\Oo(D_{\rho_1})\quad\textrm{for } j=0,1
\end{aligned}\right.
\end{equation}
in two variables $(t,x)\in\C^2$, where
\begin{itemize}
\item the coefficient $a(t,x)$ is analytic on a polydisc $D_{\rho_0,\rho_1}$ and $a(0,x)\not\equiv0$;
\item $P(t,x,X)$ and $Q(t,x,X)$ are two polynomials in $X$ with analytic coefficients on $D_{\rho_0,\rho_1}$;
\item $\widetilde{f}(t,x)\in\mathcal O(D_{\rho_1})[[t]]$.
\end{itemize}
The moment Newton polygon associated with Eq. (\ref{momentBE}) is as shown on Fig \ref{MomentNewtonPolygonmomentBE} below. In the latter, we have only shown the important points, the others being all included in the domain $C(4s_1,0)$.
\begin{center}
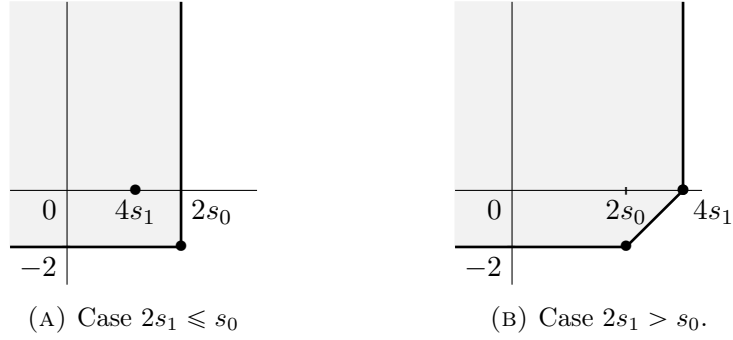
\begin{figure}[h]
	\begin{center}
		\begin{subfigure}[b]{6cm}
		\centering\begin{tikzpicture}
		\draw[fill=gray!10,color=gray!10] (-0.75,-0.75) -- (1.5,-0.75) -- (1.5,2.5) -- (-0.75,2.5) -- cycle;
		\draw(-0.75,0) -- (2.5,0);
		\draw(0,-1.25) -- (0,2.5);
		\draw (1.5,0)  node[below right]{$2s_0$};
		\draw (0.9,0) node{$\bullet$} node[below]{$4s_1$};
		\draw (0,-0.75) node{-} node[below left]{$-2$};
		\draw[line width=1pt] (-0.75,-0.75) -- (1.5,-0.75) node{$\bullet$} -- (1.5,2.5);
		\draw (0,0) node[below left]{$0$};
		\end{tikzpicture}
		\caption{Case $2s_1\leq s_0$}
		\end{subfigure}
		\begin{subfigure}[b]{6cm}
		\centering\begin{tikzpicture}[>=latex]
		\draw[fill=gray!10,color=gray!10] (-0.75,-0.75) -- (1.5,-0.75) -- (2.25,0) -- (2.25,2.5) -- (-0.75,2.5) -- cycle;
		\draw(-0.75,0) -- (2.5,0);
		\draw(0,-1.25) -- (0,2.5);
		\draw (1.5,0) node[rotate=90]{-}node[below]{$2s_0$};
		\draw (2.25,0) node{$\bullet$} node[below right]{$4s_1$};
		\draw (0,-0.75) node{-} node[below left]{$-2$};
		\draw[line width=1pt] (-0.75,-0.75) -- (1.5,-0.75) node{$\bullet$} -- (2.25,0) node{$\bullet$} -- (2.25,2.5);
		\draw (0,0) node[below left]{$0$};
		\end{tikzpicture}
		\caption{Case $2s_1>s_0$.}
		\end{subfigure}
	\end{center}
	\caption{The moment Newton polygon associated with Eq. (\ref{momentBE})}\label{MomentNewtonPolygonmomentBE}
\end{figure}
\end{center}

The critical value of Eq. (\ref{momentHE}) is then defined by
$$\sigma_c=\begin{cases}
0&\text{if }2s_1\leq s_0\\
2s_1-s_0&\text{otherwise}
\end{cases}$$
and the Gevrey regularity of the unique formal solution $\widetilde{u}(t,x)$ of Eq. (\ref{momentBE}) follows as previously from Theorem \ref{GevreyRegularityTheorem}.
\end{example}
\begin{example}
As a final example, let us look at the \textit{generalized regular moment Burgers-Korteweg-de Vries equation} (in short, the grmBKdV equation):
\begin{equation}\label{GeneralmomentBKVE}
\begin{cases}\partial_{m_0;t}u-P_{q_1}(t,x,u)\partial_{m;x}^{q_1}u-P_{q_2}(t,x,u)\partial_{m;x}^{q_2}u=\widetilde{f}(t,x)\\
u(0,x)=\varphi(x)\in\mathcal O(D_{\rho_1}),\quad q_1\geq q_2
\end{cases}
\end{equation}
in two variables $(t,x)\in\C^2$, where we assume the same conditions as before on $P_{q_1}(t,x,X)$, $P_{q_2}(t,x,X)$, $\widetilde{f}(t,x)$ and $\varphi(x)$. Denoting by $v_1$ (resp. $v_2$) the smallest valuation at $t=0$ of the coefficients of the polynomial $P_{q_1}(t,x,X)$ (resp. $P_{q_2}(t,x,X)$), the moment Newton polygon associated with Eq. (\ref{GeneralmomentBKVE}) is as shown on Fig \ref{MomentNewtonPolygonGeneralmomentBKVE} below. As was the case with the previous example,  we have only shown the important points, the others being all included, either in the domain $C(q_1s_1,v_1)$, or in the domain $C(q_2s_1,v_2)$.
\begin{center}
\begin{figure}[h]
	\begin{center}
		\begin{subfigure}[t]{3.75cm}
		\centering\begin{tikzpicture}
		\draw[fill=gray!10,color=gray!10] (-0.5,-0.5) -- (0.5,-0.5) -- (0.5,1.5) -- (-0.5,1.5) -- cycle;
		\draw(-0.5,0) -- (1.75,0); \draw(0,-0.75) -- (0,1.5);
		\draw (0,-0.5) node{-} node[below left]{\scriptsize{$-1$}}; \draw (0.5,0) node[rotate=90]{-} node[below right]{\scriptsize{$s_0$}} node[above right]{\scriptsize{$q_1s_1$}};
		\draw (0,1) node{-} node[left]{\scriptsize{$v_1$}}; \draw (0.5,1) node{$\bullet$};
		\draw (0,0.5) node{$\bullet$} node[left]{\scriptsize{$v_2$}}; \draw (0,0) node[below]{\scriptsize{$q_2s_1$}};
		\draw[line width=1pt] (-0.5,-0.5) -- (0.5,-0.5) node{$\bullet$} -- (0.5,1.5);
		\end{tikzpicture}
		\caption{\centering{Case $q_1s_1\leq s_0$}}\label{NPBKVE1}
		\end{subfigure}
		\begin{subfigure}[t]{3.75cm}
		\centering\begin{tikzpicture}
		\draw[fill=gray!10,color=gray!10] (-0.5,-0.5) -- (0.5,-0.5) -- (1,1) -- (1,1.5) -- (-0.5,1.5) -- cycle;
		\draw(-0.5,0) -- (1.75,0); \draw(0,-0.75) -- (0,1.5);
		\draw (0,-0.5) node{-} node[below left]{\scriptsize{$-1$}}; \draw (0.5,0) node[rotate=90]{-} node[below]{\scriptsize{$s_0$}};
		\draw (0,1) node{-}node[left]{\scriptsize{$v_1$}}; \draw (1,0) node[rotate=90]{-} node[below]{\scriptsize{$q_1s_1$}};
		\draw[line width=1pt] (-0.5,-0.5) -- (0.5,-0.5) node{$\bullet$} -- (1,1) node{$\bullet$} -- (1,1.5);
		\draw (0,0.5) node{$\bullet$} node[left]{\scriptsize{$v_2$}}; \draw (0,0) node[below]{\scriptsize{$q_2s_1$}};
		\end{tikzpicture}
		\caption{\centering{Case $q_2s_1\leq s_0<q_1s_1$}}\label{NPBKVE2}
		\end{subfigure}
		
		\begin{subfigure}[t]{3.75cm}
		\centering\begin{tikzpicture}
		\draw[fill=gray!10,color=gray!10] (-0.5,-0.5) -- (0.5,-0.5) -- (1,0.5) -- (1,1.5) -- (-0.5,1.5) -- cycle;
		\draw(-0.5,0) -- (1.75,0); \draw(0,-0.75) -- (0,1.5);
		\draw (0,-0.5) node{-} node[below left]{\scriptsize{$-1$}}; \draw (0.5,0) node[rotate=90]{-} node[below]{\scriptsize{$s_0$}};
		\draw (0,0.5) node{-}node[left]{\scriptsize{$v$}}; \draw (1,0) node[rotate=90]{-} node[below]{\scriptsize{$q_1s_1$}};
		\draw[line width=1pt] (-0.5,-0.5) -- (0.5,-0.5) node{$\bullet$} -- (1,0.5) node{$\bullet$} -- (1,1.5);
		\end{tikzpicture}
		\caption{\centering{Case $s_0<q_1s_1$; $q_1=q_2$; $v=\min(v_1,v_2)$}}\label{NPBKVE4}
		\end{subfigure}
		\begin{subfigure}[t]{3.75cm}
		\centering\begin{tikzpicture}
		\draw[fill=gray!10,color=gray!10] (-0.5,-0.5) -- (0.5,-0.5) -- (1.5,1) -- (1.5,1.5) -- (-0.5,1.5) -- cycle;
		\draw(-0.5,0) -- (1.75,0); \draw(0,-0.75) -- (0,1.5);
		\draw (0,-0.5) node{-} node[below left]{\scriptsize{$-1$}}; \draw (0.5,0) node[rotate=90]{-} node[below]{\scriptsize{$s_0$}};
		\draw (0,0.5) node{-}node[left]{\scriptsize{$v_2$}}; \draw (1,0) node[rotate=90]{-} node[below]{\scriptsize{$q_2s_1$}}; \draw (1,0.5) node{$\bullet$};
		\draw (0,1) node{-} node[left]{\scriptsize{$v_1$}}; \draw (1.5,0) node[rotate=90]{-} node[below]{\scriptsize{$q_1s_1$}};
		\draw[line width=1pt] (-0.5,-0.5) -- (0.5,-0.5) node{$\bullet$} -- (1.5,1) node{$\bullet$} -- (1.5,1.5);
		\end{tikzpicture}
		\caption{\centering{Case $s_0<q_2s_1<q_1s_1$ and $\dfrac{1+v_2}{q_2s_1-s_0}\geq\dfrac{1+v_1}{q_1s_1-s_0}$}}\label{NPBKVE5}
		\end{subfigure}
		\begin{subfigure}[t]{3.75cm}
		\centering\begin{tikzpicture}
		\draw[fill=gray!10,color=gray!10] (-0.5,-0.5) -- (0.5,-0.5) -- (1,0) -- (1.5,1) -- (1.5,1.5) -- (-0.5,1.5) -- cycle;
		\draw(-0.5,0) -- (1.75,0); \draw(0,-0.75) -- (0,1.5);
		\draw (0,-0.5) node{-} node[below left]{\scriptsize{$-1$}}; \draw (0.5,0) node[rotate=90]{-} node[below]{\scriptsize{$s_0$}};
		\draw (0,0) node[left]{\scriptsize{$v_2$}}; \draw (1,0) node[rotate=90]{-} node[below]{\scriptsize{$q_2s_1$}};
		\draw (0,1) node{-} node[left]{\scriptsize{$v_1$}}; \draw (1.5,0) node[rotate=90]{-} node[below]{\scriptsize{$q_1s_1$}};
		\draw[line width=1pt] (-0.5,-0.5) -- (0.5,-0.5) node{$\bullet$} -- (1,0) node{$\bullet$} -- (1.5,1) node{$\bullet$} -- (1.5,1.5);
		\end{tikzpicture}
		\caption{\centering{Case $s_0<q_2s_1<q_1s_1$ and $\dfrac{1+v_2}{q_2s_1-s_0}<\dfrac{1+v_1}{q_1s_1-s_0}$}}\label{NPBKVE6}
		\end{subfigure}
	\end{center}
	\caption{The moment Newton polygon associated with Eq. (\ref{GeneralmomentBKVE})}\label{MomentNewtonPolygonGeneralmomentBKVE}
\end{figure}
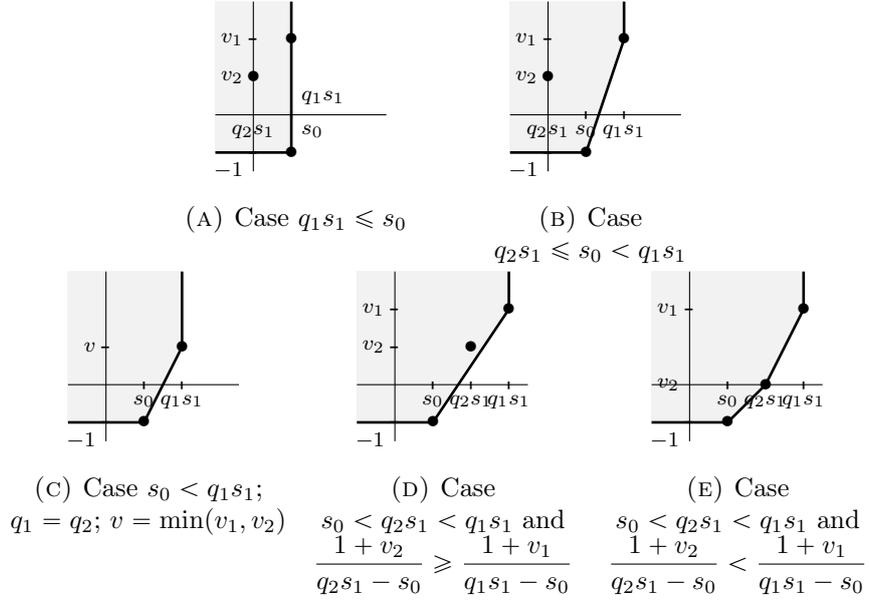
\end{center}

The critical value of Eq. (\ref{GeneralmomentBKVE}) is then given by
$$\sigma_c=\begin{cases}
0&\text{if }q_1s_1\leq s_0\\
\dfrac{q_1s_1-s_0}{1+v_1}&\text{if } q_2s_1\leq s_0<q_1s_1\\
\dfrac{q_1s_1-s_0}{1+v_1}&\text{if }s_0<q_2s_1<q_1s_1\text{ and }\dfrac{1+v_2}{q_1s_1-s_0}\leq\dfrac{1+v_2}{q_2s_1-s_0}\\
\dfrac{q_1s_1-s_0}{1+\min(v_1,v_2)}&\text{if } s_0<q_1s_1\text{ and }q_1=q_2\\
\dfrac{q_2s_1-s_0}{1+v_2}&\text{if }s_0<q_2s_1<q_1s_1\text{ and }\dfrac{1+v_2}{q_2s_1-s_0}<\dfrac{1+v_1}{q_1s_1-s_0}
\end{cases}$$
and the Gevrey regularity of the unique formal solution $\widetilde{u}(t,x)$ of Eq. (\ref{GeneralmomentBKVE}) follows again from Theorem \ref{GevreyRegularityTheorem}. In particular, this result provides the Gevrey regularity of the formal solution of the regular moment Korteweg-de Vries equation
$$\partial_{m_0;t} u+\partial_{m;x}^3 u-6u\partial_{m;x} u=\widetilde{f}(t,x),$$
and the Gevrey regularity of the formal solution of the regular moment Burgers equation
$$\partial_{m_0;t} u-\partial_{m;x}^2 u-2u\partial_{m;x} u=\widetilde{f}(t,x).$$
Indeed,  these two equations both correspond to the cases presented in Fig. \ref{NPBKVE1}, \ref{NPBKVE2} and \ref{NPBKVE5} and admit respectively the values
$$\sigma_c=\begin{cases}
0&\text{if }3s_1\leq s_0\\
3s_1-s_0&\text{otherwise}
\end{cases}\quad\text{and}\quad
\sigma_c=\begin{cases}
0&\text{if }2s_1\leq s_0\\
2s_1-s_0&\text{otherwise}
\end{cases}$$
as critical value.
\end{example}

\section{Proof of Theorem \ref{GevreyRegularityTheorem}}\label{section:proof}

The proof of Theorem \ref{GevreyRegularityTheorem} is detailed in the following two sections. The first point is the most technical and the most complicated. Its proof is based on the modified Nagumo norms, a technique of majorant series and a fixed point procedure (see Section \ref{FirstPointProof}). As for the second point, it stems both from the first one and from Proposition \ref{ex} that gives an explicit example for which $\widetilde{u}(t,x)$ is $\sigma'$-Gevrey for no $\sigma'<\sigma_c$ while $\widetilde{f}(t,x)$ is $\sigma$-Gevrey with $\sigma<\sigma_c$ (see Section \ref{SecondPointProof}).

\subsection{Proof of the first point of Theorem \ref{GevreyRegularityTheorem}}\label{FirstPointProof}
According to Propositiion \ref{GevreyAlgStructure} and Corollary \ref{StabilityGevreyMomentDerivative}, it is clear that
$$\widetilde{u}(t,x)\in\DomGn{\sigma}\Rightarrow\widetilde{f}(t,x)\in\DomGn{\sigma}.$$

Reciprocally, let us fix $\sigma\geq\sigma_c$, and let us assume that
$$\widetilde{f}(t,x)=\dsum{j\geq0}{}{\TC{f}{j}(x)\dfrac{t^j}{m_0(j)}}$$
is $\sigma$-Gevrey. By assumption (see Definition \ref{defi:gevrey-order}), there exist a radius $0<r<\min(\rho_1,...,\rho_N)$ and two positive constants $C,K>0$ such that $\m{\TC{f}{j}(x)}\leq CK^jm_0(j)\Gamma(1+\sigma j)$ for all $x\in\DCr$ and all $j\geq0$.

We must prove that the coefficients $\TC{u}{j}(x)$ of the formal solution $\widetilde{u}(t,x)$ satisfy similar inequalities. The approach we present below is analogous to the ones already developed in  \cite{BLR,R16,R17,R20} in the framework of linear partial and integro-differential equations, in \cite{R21e,R23a,R23b} in the case of nonlinear partial differential equations, and in \cite{Su21} for some linear moment partial differential equations. It is based on the modified Nagumo norms introduced in Section \ref{section:nagumo} and on a technique of majorant series.

\subsubsection{First step: some preliminary inequalities}

From relations (\ref{uj}) and (\ref{Cjn}), we first get the recurrence relations
\begin{multline*}
\dfrac{\TC{u}{j+\kappa}(x)}{m_0(j+\kappa)\Gamma(1+\sigma(j+\kappa))}=\dfrac{\TC{f}{j}(x)}{m_0(j+\kappa)\Gamma(1+\sigma(j+\kappa))}\\+
\dsum{n\in\mathcal I}{}{\dsum{(\underline{i},\underline{q},\underline{r})\in\Lambda_n}{}{\dsum{{\substack{j_0+j_1+...+\\j_{r_1+...+r_n}=j-v_{\underline{i},\underline{q},\underline{r}}}}}{}{A_{\underline{r},\underline{j},n}\TC{a}{\underline{i},\underline{q},\underline{r};j_0}(x)\times\\\dprod{\ell=1}{n}{\dprod{h=j_{r_1+...+r_{\ell-1}+1}}{j_{r_1+...+r_\ell}}{\partial_{m;x}^{q_\ell}\TC{u}{h+i_\ell}(x)}}}}}
\end{multline*}
starting with $\TC{u}{j}(x)=\varphi_j(x)$ for $j=0,...,\kappa-1$, with
$$A_{\underline{r},\underline{j},n}=\dfrac{1}{m_0(j+\kappa)\Gamma(1+\sigma(j+\kappa))}\dbinom{j}{j_0,...,j_{r_1+...+r_n}}_{m_0}.$$

Let us now consider the modified Nagumo norm of indices $((j+\kappa)\alpha_\sigma,r,s)$, where $\alpha_\sigma\in(\R^+)^N$ is the multi-index whose all components are equal to $(\sigma+s_0)(\kappa+v)$, with $v=\varsigma+\max v_{\underline{i},\underline{q},\underline{r}}$ and
$$\varsigma=\max\left(\dfrac{1-(\sigma+s_0)(\kappa+\max v_{\underline{i},\underline{q},\underline{r}})}{\sigma+s_0},\max_{(\underline{i},\underline{q},\underline{r})\in\bigcup_{n\in\mathcal I}\Lambda_n}\left(\dfrac{1}{(\sigma+s_0)\left(\kappa-\dsum{\ell=1}{n}{r_\ell i_\ell}+v_{\underline{i},\underline{q},\underline{r}}\right)}\right)\right).$$
Observe that, if the first value may be non-positive, the second value is always positive; hence, $\varsigma$ is positive. Observe also that the first value implies $\alpha_\sigma\geq1$.

Hence, from Proposition \ref{norm}:
\begin{multline*}
\dfrac{\MNN{\TC{u}{j+\kappa}}{(j+\kappa)\alpha_\sigma}{r}{s}}{m_0(j+\kappa)\Gamma(1+\sigma(j+\kappa))}\leq\dfrac{\MNN{\TC{f}{j}}{(j+\kappa)\alpha_\sigma}{r}{s}}{m_0(j+\kappa)\Gamma(1+\sigma(j+\kappa))}\\+
\dsum{n\in\mathcal I}{}{\dsum{(\underline{i},\underline{q},\underline{r})\in\Lambda_n}{}{\dsum{{\substack{j_0+j_1+...+\\j_{r_1+...+r_n}=j-v_{\underline{i},\underline{q},\underline{r}}}}}{}{A_{\underline{r},\underline{j},n}\times\\\MNN{\TC{a}{\underline{i},\underline{q},\underline{r};j_0}\dprod{\ell=1}{n}{\dprod{h=j_{r_1+...+r_{\ell-1}+1}}{j_{r_1+...+r_\ell}}{\partial_{m;x}^{q_\ell}\TC{u}{h+i_\ell}}}}{(j+\kappa)\alpha_\sigma}{r}{s}}}}.
\end{multline*}

Let us now write $(j+\kappa)\alpha_\sigma$ in the form
$$(j+\kappa)\alpha_\sigma=\left(\dsum{\ell=1}{n}{\dsum{h=j_{r_1+...+r_{\ell-1}+1}}{j_{r_1+...+r_\ell}}{(h+i_\ell)}}\right)\alpha_\sigma+\dsum{\ell=1}{n}{r_\ell q_\ell}+\alpha'_\sigma(j_0)$$
with
$$\alpha'_\sigma(j_0)=\left(j_0+\kappa-\dsum{\ell=1}{n}{r_\ell i_\ell}+v_{\underline{i},\underline{q},\underline{r}}\right)\alpha_\sigma-\dsum{\ell=1}{n}{r_\ell q_\ell}.$$
Observe here that Assumption \ref{ASS1} implies\label{alphaSigmaJ}
$$j_0+\kappa-\dsum{\ell=1}{n}{r_\ell i_\ell}+v_{\underline{i},\underline{q},\underline{r}}\geq\kappa-\dsum{\ell=1}{n}{r_\ell i_\ell}+v_{\underline{i},\underline{q},\underline{r}}>0$$
and that Proposition \ref{IneqMomentNewtonPolygon} and the definition of $\varsigma$ imply that the $d$-th component $\alpha'_{\sigma,d}(j_0)$ of $\alpha'_\sigma(j_0)$ satisfies for any $d=1,...,N$ the inequalities
\begin{align*}
\alpha'_{\sigma,d}(j_0)&\geq(1+\varsigma)(\sigma+s_0)\left(\kappa-\dsum{\ell=1}{n}{r_\ell i_\ell}+v_{\underline{i},\underline{q},\underline{r}}\right)-\dsum{\ell=1}{n}{r_\ell q_{\ell,d}}\\
&\geq\varsigma(\sigma+s_0)\left(\kappa-\dsum{\ell=1}{n}{r_\ell i_\ell}+v_{\underline{i},\underline{q},\underline{r}}\right)+s_0v_{\underline{i},\underline{q},\underline{r}}\geq1.
\end{align*}
Indeed, the order $s_d$ being $\geq1$, we have $\lambda(s q_\ell)\geq q_{\ell,d}$.

Applying then Proposition \ref{NagumoNormProduct} and Corollary \ref{NagumoNormDerivative}, we finally get
\begin{multline*}
\dfrac{\MNN{\TC{u}{j+\kappa}}{(j+\kappa)\alpha_\sigma}{r}{s}}{m_0(j+\kappa)\Gamma(1+\sigma(j+\kappa))}\leq\dfrac{\MNN{\TC{f}{j}}{(j+\kappa)\alpha_\sigma}{r}{s}}{m_0(j+\kappa)\Gamma(1+\sigma(j+\kappa))}+\\
\dsum{n\in\mathcal I}{}{\dsum{(\underline{i},\underline{q},\underline{r})\in\Lambda_n}{}{\dsum{{\substack{j_0+j_1+...+\\j_{r_1+...+r_n}=j-v_{\underline{i},\underline{q},\underline{r}}}}}{}{B_{\underline{i},\underline{q},\underline{r},\underline{j},n}(x)}}}
\end{multline*}
with
$$B_{\underline{i},\underline{q},\underline{r},\underline{j},n}(x)=B'_{\underline{i},\underline{q},\underline{r},\underline{j},n}\dprod{\ell=1}{n}{\dprod{h=j_{r_1+...+r_{\ell-1}+1}}{j_{r_1+...+r_\ell}}{\dfrac{\MNN{\TC{u}{h+i_\ell}}{(h+i_\ell)\alpha_\sigma}{r}{s}}{m_0(h+i_\ell)\Gamma(1+\sigma(h+i_\ell))}}}$$
 for all $j\geq v_{\underline{i},\underline{q},\underline{r}}$ and
\begin{multline*}
B'_{\underline{i},\underline{q},\underline{r},\underline{j},n}=\dfrac{\dbinom{j}{j_0,...,j_{r_1+...+r_n}}_{m_0}\MNN{\TC{a}{\underline{i},\underline{q},\underline{r};j_0}}{\alpha'_\sigma(j_0)}{r}{s}}{m_0(j+\kappa)\Gamma(1+\sigma(j+\kappa))}\times\\
\dprod{\ell=1}{n}{\dprod{h=j_{r_1+...+r_{\ell-1}+1}}{j_{r_1+...+r_\ell}}{\Bigg(C^{\lambda(q_\ell)}m_0(h+i_\ell)\Gamma(1+\sigma(h+i_\ell))}}\times\\
\left.\dprod{d=1}{N}{q_{\ell,d}!^{s_d}\dbinom{(h+i_\ell)(\sigma+s_0)(\kappa+v)+q_{\ell,d}-1}{q_{\ell,d}}^{s_d}}\right).
\end{multline*}

\subsubsection{Second step: bound of $B'_{\underline{i},\underline{q},\underline{r},\underline{j},n}$}
Since 
$$
\dbinom{j}{j_0,...,j_{r_1+...+r_n}}_{m_0}=\frac{m_0(j)}{m_0(j_0)\dprod{\ell=1}{n}{\dprod{h=j_{r_1}+\ldots+j_{r_{\ell-1}}+1}{j_{r_1}+\ldots+j_{r_\ell}}{m_0(h)}}},
$$
we can alternatively write
\begin{multline*}
B'_{\underline{i},\underline{q},\underline{r},\underline{j},n}=C^{\sum_{\ell=1}^{n}r_\ell\lambda(q_\ell)}\dfrac{m_0(j)\MNN{\TC{a}{\underline{i},\underline{q},\underline{r};j_0}}{\alpha'_\sigma(j_0)}{r}{s}}{m_0(j+\kappa)\Gamma(1+\sigma(j+\kappa))m_0(j_0)}\times\\
\dprod{\ell=1}{n}{\dprod{h=j_{r_1+...+r_{\ell-1}+1}}{j_{r_1+...+r_\ell}}{\Bigg(\frac{m_0(h+i_\ell)\Gamma(1+\sigma(h+i_\ell))}{m_0(h)}}}\times\\
\left.\dprod{d=1}{N}{q_{\ell,d}!^{s_d}\dbinom{(h+i_\ell)(\sigma+s_0)(\kappa+v)+q_{\ell,d}-1}{q_{\ell,d}}^{s_d}}\right).
\end{multline*}

Since $\kappa\geq1$ and $m_0$ is a regular moment function of order $s_0>0$, there exists a positive constant $C_1>0$ such that
\begin{equation}
\dfrac{m_0(j)}{m_0(j+\kappa)}=\dprod{k=0}{\kappa-1}{\dfrac{m_0(j+k)}{m_0(j+k+1)}}\leq\dfrac{C_1^\kappa}{(j+1)^{s_0}...(j+\kappa)^{s_0}}\leq\dfrac{C_1^\kappa}{(j+1)^{s_0\kappa}}.\label{eq:C_1}
\end{equation}

Let us now repeat this reasoning for $\frac{m_0(h+i_\ell)}{m_0(h)}$, with a fixed $\ell\in\{1,2,\ldots,n\}$. If $i_\ell\geq1$, there exists a positive constant $\widehat{C}_\ell$ such that
$$\dfrac{m_0(h+i_\ell)}{m_0(h)}=\dprod{k=0}{i_\ell-1}{\dfrac{m_0(j+k+1)}{m_0(j+k)}}\leq \widehat{C}_\ell^{i_\ell}(h+1)^{s_0}...(h+i_\ell)^{s_0}\leq \widehat{C}_\ell^{i_\ell}i_\ell!^{s_0}(j+1)^{s_0i_\ell}$$
for all $h=j_{r_1+\ldots+r_{\ell-1}+1},\ldots,j_{r_\ell}$ (we have indeed $h\leq j_{r_1+\ldots+r_{\ell}}\leq j$). Observe that such inequality remains valid when $i_\ell=0$. Consequently,
$$\dprod{h=j_{r_1+\ldots+r_{\ell-1}+1}}{j_{r_1+\ldots+r_{\ell}}}{\dfrac{m_0(h+i_\ell)}{m_0(h)}}\leq \widehat{C}_\ell^{r_\ell i_\ell}i_\ell!^{r_\ell s_0}(j+1)^{s_0r_\ell i_\ell},$$
and there exists a positive constant $C_2=\dprod{\ell=1}{n}{\widehat{C}_\ell^{r_\ell i_\ell}i_\ell!^{s_0r_\ell}}>0$ such that
\begin{equation}
\dprod{\ell=1}{n}{\dprod{h=j_{r_1+\ldots+r_{\ell-1}+1}}{j_{r_1+\ldots+r_{\ell}}}{\dfrac{m_0(h+i_\ell)}{m_0(h)}}}\le C_2(j+1)^{s_0\sum_{\ell=1}^n r_\ell i_\ell}\label{eq:C_2}
\end{equation}

Let us now observe that for all $j\geq v_{\underline{i},\underline{q},\underline{r}}$, we have
\begin{multline*}
\dfrac{\dprod{\ell=1}{n}{\dprod{h=j_{r_1+\ldots r_{\ell-1}+1}}{j_{r_1+\ldots r_\ell}}{\Gamma(1+\sigma(h+i_\ell))}}}{\Gamma(1+\sigma(j+\kappa))}=\dfrac{\Gamma(1+\sigma(j-v_{\underline{i},\underline{q},\underline{r}}+\sum_{\ell=1}^n r_\ell i_\ell))}{\Gamma(1+\sigma(j+\kappa))}\\
\times\dfrac{1}{\Gamma(1+\sigma j_0)}\dfrac{1}{\dbinom{\sigma(j-v_{\underline{i},\underline{q},\underline{r}}+\sum_{\ell=1}^n r_\ell i_\ell)}{\sigma j_0,\sigma (j_1+i_1),\ldots,\sigma(j_{r_1+\ldots r_n}+i_n)}}.
\end{multline*}
Applying the Stirling's Formula and Assumption \ref{ASS1}, we easily check that there exists a positive constant $C_3$ such that
\begin{equation*}
\dfrac{\Gamma(1+\sigma(j-v_{\underline{i},\underline{q},\underline{r}}+\sum_{\ell=1}^n r_\ell i_\ell))}{\Gamma(1+\sigma(j+\kappa))}\le C_3(j+1)^{-\sigma(\kappa-\sum_{\ell=1}^n r_\ell i_\ell+v_{\underline{i},\underline{q},\underline{r}})}.
\end{equation*}
Consequently, since we also have
$$\dbinom{\sigma(j-v_{\underline{i},\underline{q},\underline{r}}+\sum_{\ell=1}^n r_\ell i_\ell)}{\sigma j_0,\sigma (j_1+i_1),\ldots,\sigma(j_{r_1+\ldots r_n}+i_n)}\geq1,$$
we deduce that

\begin{equation}
\dfrac{\dprod{\ell=1}{n}{\dprod{h=j_{r_1+\ldots r_{\ell-1}+1}}{j_{r_1+\ldots r_\ell}}{\Gamma(1+\sigma(h+i_\ell))}}}{\Gamma(1+\sigma(j+\kappa))}\leq C_3\dfrac{(j+1)^{-\sigma(\kappa-\sum_{\ell=1}^n r_\ell i_\ell+v_{\underline{i},\underline{q},\underline{r}})}}{\Gamma(1+\sigma j_0)}.\label{eq:C_3}
\end{equation}

Let us also notice that for any $\ell=1,2,\ldots,n$, $h=j_{r_1+\ldots+r_{\ell-1}+1},\ldots,j_{r_1+\ldots+r_{\ell}}$ and $d=1,2,\ldots,N$ we have
\begin{multline*}
 q_{\ell,d}!^{s_d}\dbinom{(h+i_\ell)(\sigma+s_0)(\kappa+v)+q_{\ell,d}-1}{q_{\ell,d}}^{s_d}\\=\left(\frac{\Gamma((h+i_\ell)(\sigma+s_0)(\kappa+v)+q_{\ell,d})}{\Gamma((h+i_\ell)(\sigma+s_0)(\kappa+v))}\right)^{s_d}\\
 =\left(\dprod{k=0}{q_{\ell,d}-1}{\Big((h+i_\ell)(\sigma+s_0)(\kappa+v)+k\Big)\right)^{s_d}}.
\end{multline*}
Moreover, since $s_d\ge 1$ for all $d=1,2,\ldots,N$ and $r_\ell\ge 1$ for every $\ell=1,2,\ldots,n$, it follows from Proposition \ref{IneqMomentNewtonPolygon} that 
\begin{equation*}
 (\sigma+s_0)(\kappa+v)\ge(\sigma+s_0)\dsum{\ell=1}{n}{r_\ell i_\ell}+s_0v_{\underline{i},\underline{q},\underline{r}}+\dsum{\ell=1}{n}{r_\ell\lambda(sq_\ell)}\ge q_{\ell,d}.
\end{equation*}
Hence, $k\le (\sigma+s_0)(\kappa+v)$ for every $k=0,1,\ldots,q_{\ell,d}-1$ and we receive
\begin{multline*}
 \left(\dprod{k=0}{q_{\ell,d}-1}{\Big((h+i_\ell)(\sigma+s_0)(\kappa+v)+k\Big)\right)^{s_d}}\\\le \Big((\sigma+s_0)(\kappa+v)(h+i_\ell+1)\Big)^{s_dq_{\ell,d}}\\\le\Big((\sigma+s_0)(\kappa+v)(i_\ell+1)\Big)^{s_dq_{\ell,d}}(h+1)^{s_dq_{\ell,d}}
\\\le\Big((\sigma+s_0)(\kappa+v)(i_\ell+1)\Big)^{s_dq_{\ell,d}}(j+1)^{s_dq_{\ell,d}}
\end{multline*}
Hence, there exists a positive constant
$$C_4=\dprod{\ell=1}{n}{\Big((\sigma+s_0)(\kappa+v)(i_\ell+1)\Big)^{r_\ell\lambda(sq_{\ell})}}>0$$
such that
\begin{multline}
 \dprod{\ell=1}{n}{\dprod{h=j_{r_1+\ldots+r_{\ell-1}+1}}{j_{r_1+\ldots+r_{\ell}}}{\dprod{d=1}{N}{q_{\ell,d}!^{s_d}\dbinom{(h+i_\ell)(\sigma+s_0)(\kappa+v)+q_{\ell,d}-1}{q_{\ell,d}}^{s_d}}}}\\\leq C_4(j+1)^{\sum_{\ell=1}^n r_\ell\lambda(sq_\ell)}
 \label{eq:C_4}
\end{multline}

Combining results from (\ref{eq:C_1}), (\ref{eq:C_2}), (\ref{eq:C_3}) and (\ref{eq:C_4}) we finally receive that there exists a positive constant $C_5>0$ such that
\begin{multline*}
 B'_{\underline{i},\underline{q},\underline{r},\underline{j},n}\le \dfrac{C_5\MNN{\TC{a}{\underline{i},\underline{q},\underline{r};j_0}}{\alpha'_\sigma(j_0)}{r}{s}}{\Gamma(1+\sigma j_0)m_0(j_0)}(j+1)^{(\sigma+s_0)(\sum_{\ell=1}^n r_\ell i_\ell -\kappa)-\sigma v_{\underline{i},\underline{q},\underline{r}}+\sum_{\ell=1}^n r_\ell\lambda(sq_\ell)}
\end{multline*}
From Proposition \ref{IneqMomentNewtonPolygon} we further receive an inequality
\begin{multline*}
(\sigma+s_0)\left(\sum_{\ell=1}^n r_\ell i_\ell -\kappa-v_{\underline{i},\underline{q},\underline{r}}\right)+s_0v_{\underline{i},\underline{q},\underline{r}}+\sum_{\ell=1}^n r_\ell\lambda(sq_\ell)\\\le -s_0v_{\underline{i},\underline{q},\underline{r}}-\sum_{\ell=1}^n r_\ell\lambda(sq_\ell)+s_0v_{\underline{i},\underline{q},\underline{r}}+\sum_{\ell=1}^n r_\ell\lambda(sq_\ell)=0,
\end{multline*}
from which it follows that
\begin{equation}
 B'_{\underline{i},\underline{q},\underline{r},\underline{j},n}\le \dfrac{C_5\MNN{\TC{a}{\underline{i},\underline{q},\underline{r};j_0}}{\alpha'_\sigma(j_0)}{r}{s}}{\Gamma(1+\sigma j_0)m_0(j_0)}.\label{eq:B'}
\end{equation}

Using (\ref{eq:B'}) we further conclude that
\begin{multline*}
B_{\underline{i},\underline{q},\underline{r},\underline{j},n}(x)\le \dfrac{C_5\MNN{\TC{a}{\underline{i},\underline{q},\underline{r};j_0}}{\alpha'_\sigma(j_0)}{r}{s}}{\Gamma(1+\sigma j_0)m_0(j_0)}\times\\\dprod{\ell=1}{n}{\dprod{h=j_{r_1+...+r_{\ell-1}+1}}{j_{r_1+...+r_\ell}}{\dfrac{\MNN{\TC{u}{h+i_\ell}}{(h+i_\ell)\alpha_\sigma}{r}{s}}{m_0(h+i_\ell)\Gamma(1+\sigma(h+i_\ell))}}}
\end{multline*}
for all $j\geq v_{\underline{i},\underline{q},\underline{r}}$, and consequently
\begin{multline}\label{eq:MNNuj}
\dfrac{\MNN{\TC{u}{j+\kappa}}{(j+\kappa)\alpha_\sigma}{r}{s}}{m_0(j+\kappa)\Gamma(1+\sigma(j+\kappa))}\leq\dfrac{\MNN{\TC{f}{j}}{(j+\kappa)\alpha_\sigma}{r}{s}}{m_0(j+\kappa)\Gamma(1+\sigma(j+\kappa))}
\\+
\dsum{n\in\mathcal I}{}{\dsum{(\underline{i},\underline{q},\underline{r})\in\Lambda_n}{}{\dsum{{\substack{j_0+j_1+...+\\j_{r_1+...+r_n}=j-v_{\underline{i},\underline{q},\underline{r}}}}}{}{\dfrac{C_5\MNN{\TC{a}{\underline{i},\underline{q},\underline{r};j_0}}{\alpha'_\sigma(j_0)}{r}{s}}{\Gamma(1+\sigma j_0)m_0(j_0)}\times\\\dprod{\ell=1}{n}{\dprod{h=j_{r_1+...+r_{\ell-1}+1}}{j_{r_1+...+r_\ell}}{\dfrac{\MNN{\TC{u}{h+i_\ell}}{(h+i_\ell)\alpha_\sigma}{r}{s}}{m_0(h+i_\ell)\Gamma(1+\sigma(h+i_\ell))}}}}}}
\end{multline}

We shall now bound the modified Nagumo norms $\MNN{\TC{u}{j}}{j\alpha_\sigma}{r}{s}$ for any $j\geq0$. To do that, we shall use the classical majorant series method.

\subsubsection{Third step: the majorant series method}
First of all, let us set
$$g_{j,s}=\dfrac{\MNN{\TC{f}{j}}{(j+\kappa)\alpha_\sigma}{r}{s}}{m_0(j+\kappa)\Gamma(1+\sigma(j+\kappa))}\quad\text{and}\quad\alpha_{\underline{i},\underline{q},\underline{r},j,s}=\dfrac{C_5\MNN{\TC{a}{\underline{i},\underline{q},\underline{r};j}}{\alpha'_\sigma(j)}{r}{s}}{\Gamma(1+\sigma j)m_0(j)},$$
and let us prove the following technical lemma.

\begin{lemma}\label{MajorantGjsAlphaiqrjs}
There exist four positive constants $B',B'',C',C''>0$ such that the following inequalities hold for all $j\geq0$:
$$g_{j,s}\leq C' B'^j\quad\text{and}\quad\alpha_{\underline{i},\underline{q},\underline{r},j,s}\leq C''B''^j.$$
\end{lemma}

\begin{proof}
From Corollary \ref{NagumoNormProductWithOne}, we first deduce the inequality
$$g_{j,s}\leq\dfrac{\MNN{\TC{f}{j}}{j\alpha_\sigma}{r}{s}}{m_0(j)\Gamma(1+\sigma j)}\times\dfrac{r^{\lambda(\kappa\alpha_\sigma)}\Gamma(1+\sigma\kappa)}{\dbinom{\sigma(j+\kappa)}{\sigma j}}\times\dfrac{m_0(j)}{m_0(j+\kappa)}.$$
The sought inequality follows then from Proposition \ref{prop:gevrey_norm}, inequality (\ref{eq:C_1}) and the fact that $\dbinom{\sigma(j+\kappa)}{\sigma j}\geq1$.

The second inequality on $\alpha_{\underline{i},\underline{q},\underline{r},j,s}$ is proved in a similar way (we use the fact that $\TC{a}{\underline{i},\underline{q},\underline{r};j}(x)$ is analytic on $\DCn$; hence $0$-Gevrey, and calculations from page \pageref{alphaSigmaJ} to check that $\alpha'_\sigma(j)-j\alpha_\sigma\geq1$ in order to apply Corollary \ref{NagumoNormProductWithOne}).
\end{proof}

Let us  now consider the formal power series $v(X)=\displaystyle\sum_{j\geq0}v_jX^j$, the coefficients of which are recursively determined for all $j\geq0$ by the relations
\begin{equation}\label{vj}
v_{j+\kappa}=g_{j,s}+\dsum{n\in\mathcal I}{}{\dsum{(\underline{i},\underline{q},\underline{r})\in\Lambda_n}{}{\dsum{{\substack{j_0+j_1+...+j_{\widetilde{r}}\\=j+\sum_{\ell=1}^{n}r_\ell i_\ell-v_{\underline{i},\underline{q},\underline{r}}}}}{}{\alpha_{\underline{i},\underline{q},\underline{r},j_0,s}v_{j_1}...v_{j_{\widetilde{r}}}}}}
\end{equation}
starting with the initial conditions
\begin{align*}
v_0&=1+\dfrac{\MNN{\varphi_0}{0}{r}{s}}{m_0(0)}\text{, and, for $j=1,...,\kappa-1$ (if $\kappa\geq2$):}\\
v_j&=\dfrac{\MNN{\varphi_j}{j\alpha_\sigma}{r}{s}}{m_0(j)\Gamma(1+\sigma j)}+\sum_{(\underline{i},\underline{q},\underline{r})\in V_j}\sum_{\substack{j_0+j_1+...+j_{\widetilde{r}}\\=j-\kappa+\sum_{\ell=1}^{n}r_\ell i_\ell-v_{\underline{i},\underline{q},\underline{r}}}}\alpha_{\underline{i},\underline{q},\underline{r},j_0,s}v_{j_1}...v_{j_{\widetilde{r}}},
\end{align*}
where
$$\widetilde{r}=\max_{(\underline{i},\underline{q},\underline{r})\in\bigcup_{n\in\mathcal I}\Lambda_n}(r_1+...+r_n),$$
and where
$$V_j=\left\{(\underline{i},\underline{q},\underline{r})\in\bigcup_{n\in\mathcal I}\Lambda_n\text{ such that }j-\kappa+\dsum{\ell=1}{n}{r_\ell i_\ell}-v_{\underline{i},\underline{q},\underline{r}}\geq0\right\}.$$
Observe that Assumption \ref{ASS1} implies
$$j-\kappa+\dsum{\ell=1}{n}{r_\ell i_\ell}-v_{\underline{i},\underline{q},\underline{r}}<j;$$
hence, the initial conditions on the $v_j$'s with $j=1,...,\kappa-1$ make sense.

\begin{proposition}\label{MS}
The inequalities
\begin{equation}\label{MSj}
0\leq\dfrac{\MNN{\TC{u}{j}}{j\alpha_\sigma}{r}{s}}{m_0(j)\Gamma(1+\sigma j)}\leq v_j
\end{equation}
hold for all $j\geq0$.
\end{proposition}

\begin{proof}
According to the initial conditions on the $u_j$'s and on the $v_j$'s, the inequalities (\ref{MSj}) hold for all $j=0,...,\kappa-1$. Let us now suppose that these inequalities are true for all $k\leq j-1+\kappa$ for a certain $j\geq0$, and let us prove them for $j+\kappa$.

First of all, applying our hypotheses to relations (\ref{eq:MNNuj}), we have
\begin{multline}\label{FirstIneq}
0\leq\dfrac{\MNN{\TC{u}{j+\kappa}}{(j+\kappa)\alpha_\sigma}{r}{s}}{m_0(j)\Gamma(1+\sigma(j+\kappa))}\leq g_{j,s}+\\
\dsum{n\in\mathcal I}{}{\dsum{(\underline{i},\underline{q},\underline{r})\in\Lambda_n}{}{\dsum{{\substack{j'_0+j'_1+...+\\j'_{r_1+...+r_n}=j-v_{\underline{i},\underline{q},\underline{r}}}}}{}{\alpha_{\underline{i},\underline{q},\underline{r},j'_0,s}\dprod{\ell=1}{n}{\dprod{h=j'_{r_1+...+r_{\ell-1}+1}}{j'_{r_1+...+r_\ell}}{v_{h+i_\ell}}}}}}
\end{multline}
and then
\begin{multline}\label{SecondIneq}
0\leq\dfrac{\MNN{\TC{u}{j+\kappa}}{(j+\kappa)\alpha_\sigma}{r}{s}}{m_0(j)\Gamma(1+\sigma(j+\kappa))}\leq g_{j,s}+\\
\dsum{n\in\mathcal I}{}{\dsum{(\underline{i},\underline{q},\underline{r})\in\Lambda_n}{}{\dsum{{\substack{j_0+j_1+...+j_{r_1+...+r_n}\\=j+\sum_{\ell=1}^{n}r_\ell i_\ell-v_{\underline{i},\underline{q},\underline{r}}}}}{}{\alpha_{\underline{i},\underline{q},\underline{r},j_0,s}v_{j_1}...v_{j_{r_1+...+r_n}}}}}
\end{multline}
since all the tuples $(j_0',j'_1,...,j'_{r_1+...+r_n})$ in (\ref{FirstIneq}) satisfy
$$\dsum{\ell=1}{n}{\dsum{h=j'_{r_1+...+r_{\ell-1}+1}}{j'_{r_1+...+r_\ell}}{(h+i_\ell)}}=j+\sum_{\ell=1}^{n}r_\ell i_\ell-v_{\underline{i},\underline{q},\underline{r}}$$
and since all the terms $\alpha_{i,q,p,j_0,s}v_{j_1}...v_{j_{r_1+...+r_n}}$ in (\ref{SecondIneq}) are nonnegative.

Next, let us observe that any tuple $(j_0,...,j_{r_1+...+r_n})\in\mathbb N^{r_1+...+r_n+1}$ such that $j_0+...+j_{r_1+...+r_n}=j+\sum_{\ell=1}^{n}r_\ell i_\ell-v_{\underline{i},\underline{q},\underline{r}}$ can be seen as the tuple $(j_0,...,j_{r_1+...+r_n},j_{r_1+...+r_n+1},...,j_{\widetilde{r}})\in\mathbb N^{\widetilde{r}+1}$, where $j_{r_1+...+r_n+1}=...=j_{\widetilde{r}}=0$. Therefore, using the fact that $v_0\geq1$, we have
\begin{align*}
0\leq\alpha_{\underline{i},\underline{q},\underline{r},j_0,s}v_{j_1}...v_{j_{r_1+...+r_n}}&\leq\alpha_{\underline{i},\underline{q},\underline{r},j_0,s}v_{j_1}...v_{j_{r_1+...+r_n}}v_0^{\widetilde{r}-r_1-...-r_n}\\
&=\alpha_{\underline{i},\underline{q},\underline{r},j_0,s}v_{j_1}...v_{j_{\widetilde{r}}},
\end{align*}
and, consequently, the inequalities
\begin{align*}
0&\leq\dsum{{\substack{j_0+j_1+...+j_{r_1+...+r_n}\\=j+\sum_{\ell=1}^{n}r_\ell i_\ell-v_{\underline{i},\underline{q},\underline{r}}}}}{}{\alpha_{\underline{i},\underline{q},\underline{r},j_0,s}v_{j_1}...v_{j_{r_1+...+r_n}}}\\
&\leq\dsum{{\substack{j_0+j_1+...+j_{r_1+...+r_n}+0+...+0\\=j+\sum_{\ell=1}^{n}r_\ell i_\ell-v_{\underline{i},\underline{q},\underline{r}}}}}{}{\alpha_{\underline{i},\underline{q},\underline{r},j_0,s}v_{j_1}...v_{j_{\widetilde{r}}}}\\
&\leq\dsum{{\substack{j_0+j_1+...+j_{\widetilde{r}}\\=j+\sum_{\ell=1}^{n}r_\ell i_\ell-v_{\underline{i},\underline{q},\underline{r}}}}}{}{\alpha_{\underline{i},\underline{q},\underline{r},j_0,s}v_{j_1}...v_{j_{\widetilde{r}}}}
\end{align*}
hold, since all the terms are nonnegative.

Hence, the relations
\begin{multline*}
0\leq\dfrac{\MNN{\TC{u}{j+\kappa}}{(j+\kappa)\alpha_\sigma}{r}{s}}{m_0(j)\Gamma(1+\sigma(j+\kappa))}\leq g_{j,s}+\\
\dsum{n\in\mathcal I}{}{\dsum{(\underline{i},\underline{q},\underline{r})\in\Lambda_n}{}{\dsum{{\substack{j_0+j_1+...+j_{\widetilde{r}}\\=j+\sum_{\ell=1}^{n}r_\ell i_\ell-v_{\underline{i},\underline{q},\underline{r}}}}}{}{\alpha_{\underline{i},\underline{q},\underline{r},j_0,s}v_{j_1}...v_{j_{\widetilde{r}}}}}}=v_{j+\kappa}
\end{multline*}
which ends the proof of Proposition \ref{MS}.
\end{proof}

The following Proposition \ref{boundvj} allows us to bound the $v_j$'s.

\begin{proposition}\label{boundvj}
The formal series $v(X)$ is convergent. In particular, there exist two positive constants $C',K'>0$ such that $v_j\leq C'K'^j$ for all $j\geq0$.
\end{proposition}

\begin{proof}
It is sufficient to prove the convergence of $v(X)$.

First of all, let us start by observing that $v(X)$ is the unique formal power series in $X$ solution of the functional equation
\begin{equation}\label{func}
v(X)=X\alpha (X)(v(X))^{\widetilde{r}}+h(X),
\end{equation}
where $\alpha(X)$ and $h(X)$ are the two formal power series defined by
\begin{align*}
\alpha(X)&=\dsum{n\in\mathcal I}{}{\dsum{(\underline{i},\underline{q},\underline{r})\in\Lambda_n}{}{X^{\kappa-\sum_{\ell=1}^{n}r_\ell i_\ell-1+v_{\underline{i},\underline{q},\underline{r}}}\alpha_{\underline{i},\underline{q},\underline{r},s}(X)}}\text{ and}\\
h(X)&=A_0+A_1 X+...+A_{\kappa-1}X^{\kappa-1}+X^{\kappa}\displaystyle\sum_{j\geq0}g_{j,s}X^j
\end{align*}
with
\begin{align*}
&\alpha_{\underline{i},\underline{q},\underline{r},s}(X)=\sum_{j\geq0}\alpha_{\underline{i},\underline{q},\underline{r},j,s}X^j,\\
&A_0=1+\dfrac{\MNN{\varphi_0}{0}{r}{s}}{m_0(0)}\text{, and}\\
&A_j=\dfrac{\MNN{\varphi_j}{j\alpha_\sigma}{r}{s}}{m_0(j)\Gamma(1+\sigma j)}\text{for $j=1,...,\kappa-1$ (if $\kappa\geq2$).}
\end{align*}
Observe that we have $\kappa-\sum_{\ell=1}^{n}r_\ell i_\ell-1+v_{\underline{i},\underline{q},\underline{r}}\geq0$ from Assumption \ref{ASS1}.

From Lemma \ref{MajorantGjsAlphaiqrjs} it follows that both $\alpha(X)$ and $h(X)$ are convergent power series with nonnegative coefficients, with radii of convergence $r_\alpha$ and $r_h$, respectively. It follows then that they both define increasing functions within their respective regions of convergence. Moreover, seeing as $\TC{a}{\underline{i},\underline{q},\underline{r};0}(x)\not\equiv 0$ and $A_0\geq1$, we have $\alpha(r)>0$ and $h(r)>0$ for all $r\in]0,r_\alpha[$ and $r\in]0,r_h[$ respectively.

To determine that $v(X)$ is convergent, the fixed point method will be used. Let us define a formal power series $V(X)=\dsum{\mu\ge 0}{}{V_\mu(X)}$ and let us choose the solution of the functional equation (\ref{func}) given by the system
$$
\begin{cases}
 V_0(X)=h(X)\\
 V_{\mu+1}(X)=X\alpha(X)\dsum{\mu_1+\ldots+\mu_{\widetilde{r}}=\mu}{}{V_{\mu_1}(X)\dots V_{\mu_{\widetilde{r}}}(X)}\quad\textrm{ for }\mu\ge 0.
\end{cases}
$$
By inductive reasoning on $\mu\geq0$, we establish that
$$
V_\mu(x)=\widetilde{C}_{\mu,\widetilde{r}}X^\mu\alpha(X)^\mu h(X)^{(\widetilde{r}-1)\mu+1}
$$
with 
$$
\widetilde{C}_{\mu+1,\widetilde{r}}=\dsum{\mu_1+\ldots+\mu_{\widetilde{r}}=\mu}{×}{\widetilde{C}_{\mu_1,\widetilde{r}}\dots\widetilde{C}_{\mu_{\widetilde{r}},\widetilde{r}}}
$$
for every $\mu>0$ and $\widetilde{C}_{0,\widetilde{r}}=1$.

Directly from this representation, it follows from the analyticity of $\alpha(X)$ and $h(X)$ that all the $V_\mu(X)$ define analytic functions on the disc with center $0\in\C$ and radius $\min\{r_\alpha,r_h\}$). Moreover, for all $\mu\geq0$, the function $V_\mu(X)$ is of order $X^\mu$. Hence, the series $V(X)$ makes sense as a formal power series in $X$, and we obtain $V(X)=v(X)$ by unicity.

To conclude the proof, it remains to show that $V(X)$ is convergent. To do that, let us fix $0<r<\min\{r_\alpha,r_h\}$. Then, for all $\mu\ge 0$ and for $|X|\le r$ we receive
$$
|V_\mu(X)|\le \widetilde{C}_{\mu,\widetilde{r}} |X|^\mu\alpha(r)^{\mu}h(r)^{(\widetilde{r}-1)\mu+1}.
$$
Moreover, notice that, since $\widetilde{C}_{\mu,\widetilde{r}}$ are generalized Catalan numbers\footnote{These numbers were named in honor of the mathematician Eug\`{e}ne Charles Catalan (1814-1894). They appear in many probabilist, graphs and combinatorial problems. For example, they can be seen as the number of $(p+1)$-ary trees with $j$ source-nodes, or as the number of ways of associating $j$ applications of a given $(p+1)$-ary operation, or as the number of ways of subdividing a convex polygon into $j$ disjoint ($p+2$)-gons by means of non-intersecting diagonals. They also appear in theoretical computers through the generalized Dyck words. See for instance \cite{HP91} and the references inside.} (see for instance \cite{HP91,Kl70,PoSz54}), we have the bound $\widetilde{C}_{\mu,\widetilde{r}}\le 2^{\widetilde{r}\mu}$ for all $\mu\geq0$. Hence,
$$
|V_\mu(X)|\le h(r)\left(2^{\widetilde{r}} \alpha(r)h(r)^{(\widetilde{r}-1)}|X|\right)^\mu,
$$
and the series $V(X)$ is normally convergent on any disc with center $0\in\C$ and radius
$$0<r'<\min\left(r,\frac{1}{2^{\widetilde{r}} \alpha(r)h(r)^{(\widetilde{r}-1)}}\right).$$
From this, it follows that $V(X)$ is analytic at $0\in\C$, which achieves the proof of Proposition \ref{boundvj}.
\end{proof}

According to Propositions \ref{MS} and \ref{boundvj}, we can now bound the modified Nagumo norms $\MNN{\TC{u}{j}}{j\alpha_\sigma}{r}{s}$.

\begin{corollary}\label{majNN}
Let $C',K'>0$ be as in Proposition \ref{boundvj}. Then, the following inequality holds for all $j\geq0$:
$$\MNN{\TC{u}{j}}{j\alpha_\sigma}{r}{s}\leq C'K'^jm_0(j)\Gamma(1+\sigma j).$$
\end{corollary}

We are now able to conclude the proof of the first point of Theorem \ref{GevreyRegularityTheorem}.

\subsubsection{Fourth step: conclusion}
We must prove that the sup-norm of the $\TC{u}{j}(x)$ has estimates similar to the ones on the norms $\MNN{\TC{u}{j}}{j\alpha_\sigma}{r}{s}$ (see Corollary \ref{majNN}). To this end, we proceed by shrinking the polydisc $D_{r,...,r}$. Let us choose $0<\rho<r$ and let us apply Proposition \ref{prop:sup_norm}: there exists a positive constant $A>0$ such that the following inequality holds for all $j\geq0$ and all $x\in D_{\rho,...,\rho}$:
$$\m{\TC{u}{j}(x)}\leq A^{\lambda(j\alpha_\sigma)}\MNN{\TC{u}{j}}{j\alpha_\sigma}{r}{s}.$$
Observing then that $\lambda(j\alpha_\sigma)=j\lambda(\alpha_\sigma)$, we finally deduce from Corollary \ref{majNN} that
$$\m{\TC{u}{j}(x)}\leq C' (K'A^{\lambda(\alpha_\sigma)})^jm_0(j)\Gamma(1+\sigma j)$$
for all $x\in D_{\rho,...,\rho}$ and all $j\geq0$., which ends the proof of the first point of Theorem \ref{GevreyRegularityTheorem}.

To conclude the proof of Theorem \ref{GevreyRegularityTheorem}, it remains to show that its second point also holds.

\subsection{Proof of the second point of Theorem \ref{GevreyRegularityTheorem}}\label{SecondPointProof}

In this section, we assume $\mathcal S\neq\emptyset$ and we fix $0\leq\sigma<\sigma_c$ (of course, this case does not occur when $\mathcal S=\emptyset$). 

According to the filtration of the $\sigma$-Gevrey spaces $\DomGn{s}$ (see (\ref{FiltrationGevreySpaces})) and the first point of Theorem \ref{GevreyRegularityTheorem}, it is clear that we have the following implications:
\begin{align*}
\widetilde{f}(t,x)\in\DomGn{\sigma}&\Rightarrow\widetilde{f}(t,x)\in\DomGn{\sigma_c}\\
&\Rightarrow\widetilde{u}(t,x)\in\DomGn{\sigma_c}.
\end{align*}

Therefore, to conclude that we can not say better about the Gevrey order of $\widetilde{u}(t,x)$, that is $\widetilde{u}(t,x)$ is \textit{generically} $\sigma_c$-Gevrey, we need to find an example for which the formal solution $\widetilde{u}(t,x)$ of Eq. (\ref{EQ1}) is $\sigma'$-Gevrey for no $\sigma'<\sigma_c$. In Proposition \ref{ex} below, we propose a much more general example.

Before stating this, let us begin by introducing an interesting auxiliary function. Since the functions $m_1,\ldots,m_N$ are regular moment functions of respective orders $s_1,\ldots,s_N\geq1$, there exist positive constants $a_1,...,a_N>0$ such that
$$\dfrac{m_d(j+1)}{m_d(j)}\geq a_d(j+1)^{s_d}\ \textrm{ for all }j\ge 0,$$
with $d=1,\ldots,N$.

\begin{lemma}\label{ExFunction}
The function
$$\mathcal E_{m}(x)=\dprod{d=1}{N}{\left(\dsum{j_d\geq0}{}{a_d^{j_d}j_d!^{s_d}\dfrac{x_d^{j_d}}}{m_d(j_d)}\right)}$$
defines an analytic function on the polydisc $D_{1,...,1}$.
\end{lemma}

\begin{proof}
Setting $\alpha_{j_d}(x_d)= a_d^{j_d}j_d!^{s_d}\dfrac{j_d^{j_d}}{m_d(j_d)}$ for all $d=1,...,N$ and all $j_d\geq0$, we get
$$\m{\dfrac{\alpha_{j_d+1}(x_d)}{\alpha_{j_d}(x_d)}}=a_d(j_d+1)^{s_d}\dfrac{m_d(j_d)}{m_d(j_d+1)}\m{x_d}\leq\m{x_d}$$
and the result follows from the d'Alembert's Rule since the series $\dsum{j\geq0}{}{\m{x}^{j}}$ converges for all $\m{x}<1$.
\end{proof}

\begin{proposition}\label{ex}
Let us consider the equation
\begin{equation}\label{exEq}
\begin{cases}
\partial_{m_0;t}^\kappa u - \dsum{n\in\mathcal I}{}{\dsum{(\underline{i},\underline{q},\underline{r})\in\Lambda_n}{}{t^{v_{\underline{i},\underline{q},\underline{r}}}a_{\underline{i},\underline{q},\underline{r}}\left(\partial_{m_0;t}^{i_1}\partial_{m;x}^{q_1}u\right)^{r_1}...\left(\partial_{m_0;t}^{i_n}\partial_{m;x}^{q_n}u\right)^{r_n}}}=\widetilde{f}(t,x)\\
\partial_{m_0;t}^ju(t,x)|_{t=0}=\varphi_j(x),\ j=0,...,\kappa-1
\end{cases}
\end{equation}
where
\begin{itemize}
\item the coefficients $a_{\underline{i},\underline{q},\underline{r}}$ are positive real numbers for all $(\underline{i},\underline{q},\underline{r})\in\Lambda_n$ and all $n\in\mathcal I$;
\item$i_\ell^*=0$ and $q^*_\ell=(0,...,0)$ for all $\ell\in\{1,...,n^*-1\}$;
\item$r^*_{n^*}=1$;
\item the initial condition $\varphi_{i^*_{n^*}}(x)$ is the analytic function $\mathcal E_m(x)$ on the disc $D_{1,...1}$ defined in Lemma \ref{ExFunction};
\item the initial conditions $\varphi_j(x)$ for $j\neq i^*_{n^*}$ are analytic functions on $D_{1,...,1}$ satisfying $\partial_{m;x}^{\ell}\varphi_j(0)>0$ for all $\ell\in\mathbb N^N$.
\end{itemize} Suppose also that the inhomogeneity $\widetilde{f}(t,x)$ satisfies the following conditions:
\begin{itemize}
\item$\widetilde{f}(t,x)$ is $\sigma$-Gevrey;
\item$\partial_{m;x}^{\ell}\TC{f}{j}(0)\geq0$ for all $j\geq0$ and all $\ell\in\mathbb N^N$.
\end{itemize}
Then, the formal solution $\widetilde{u}(t,x)$ of Eq. (\ref{exEq}) is exactly $\sigma_c$-Gevrey.
\end{proposition}

\begin{remark}
Due to our assumptions, Eq. (\ref{exEq}) is reduced to a nonlinear equation of the form
\begin{equation*}
\begin{cases}
\partial_{m_0;t}^\kappa u - \dsum{i\in\mathcal K}{}{\dsum{q\in Q_i}{}{\left(\dsum{r\in P_{i,q}}{}{a_{i,q,r}t^{v_{i,q,r}}u^r}\right)\partial_{m_0;t}^i\partial_{m;x}^qu}}=\widetilde{f}(t,x)\\
\partial_{m_0;t}^ju(t,x)|_{t=0}=\varphi_j(x),\ j=0,...,\kappa-1
\end{cases}
\end{equation*}
where
\begin{itemize}
\item$\mathcal K$ is a nonempty subset of $\{0,...,\kappa-1\}$;
\item$Q_i$ is a nonempty finite subset of $\Na^N$ for all $i\in\mathcal K$;
\item$P_{i,q}$ is a nonempty finite subset of $\Na$ for all $i\in\mathcal K$ and all $q\in Q_i$.
\end{itemize}
However, for the sake of clarity, we retain the notations used throughout this article and will not use this simpler form. Observe in particular that we have
$$\sigma_c=\dfrac{s_0 i^*_{n^*}+\lambda(sq^*_{n^*})-s_0\kappa}{\kappa+v_{\underline{i}^*,\underline{q}^*,\underline{r}^*}-i^*_{n^*}}.$$
\end{remark}

\begin{proof}
Due to the calculations above, it is sufficient to prove that $\widetilde{u}(t,x)$ is $\sigma'$-Gevrey for no $\sigma'<\sigma_c$.

First of all, let us rewrite the general relations (\ref{uj}) as the identities
\begin{multline*}
\TC{u}{j+\kappa}(x)=A_{\underline{i}^*,\underline{q}^*,\underline{r}^*}(x)\frac{m_0(j)}{m_0(j-v_{\underline{i}^*,\underline{q}^*,\underline{r}^*})}\partial_{m;x}^{q^*_{n^*}}\TC{u}{j-v_{\underline{i}^*,\underline{q}^*,\underline{r}^*}+i^*_{n^*}}(x)\\+R_j(x)\end{multline*}
with
$$A_{\underline{i}^*,\underline{q}^*,\underline{r}^*}(x)=a_{\underline{i}^*,\underline{q}^*,\underline{r}^*}\dprod{\ell=1}{n^*-1}{\left(\TC{u}{0}(x)\right)^{r^*_\ell}}$$
and
\begin{multline*}
R_j(x)=\TC{f}{j}(x)\\+
\dsum{\substack{j_1+...+j_{r^*_1+...+r_{n^*}^*}=j-v_{\underline{i^*},\underline{q^*},\underline{r^*}}\\(j_1,...,j_{r^*_1+...+r_{n^*}^*})\neq(0,...,0,j-v_{\underline{i^*},\underline{q^*},\underline{r^*}})}}{}{C_{\underline{i^*},\underline{q^*},\underline{r^*},\underline{j},n^*}(x)}\\+
\dsum{\substack{(\underline{i},\underline{q},\underline{r})\in\bigcup_{n\in\mathcal I}\Lambda_n\\(n,\underline{i},\underline{q},\underline{r})\neq(n^*,\underline{i^*},\underline{q^*},\underline{r^*})}}{}{\dsum{{j_0+j_1+...+j_{r_1+...+r_n}=j-v_{\underline{i},\underline{q},\underline{r}}}}{}{C_{\underline{i},\underline{q},\underline{r},\underline{j},n}(x)}}
\end{multline*}
for all $j\geq0$, together with the initial conditions $\TC{u}{j}(x)=\varphi_j(x)$ for $j=0,...,\kappa-~1$. Using then our hypotheses on the coefficients $a_{\underline{i},\underline{q},\underline{r}}$, on the initial conditions $\varphi_j(x)$, and on the inhomogeneity $\widetilde{f}(t,x)$, we easily check that, for all $j\geq0$:
\begin{multline*}
\TC{u}{j(v_{\underline{i}^*,\underline{q}^*,\underline{r}^*}+\kappa-i_{n^*}^*)+i^*_{n^*}}(x)=
\left(A_{\underline{i}^*,\underline{q}^*,\underline{r}^*}(x)\right)^j\partial_{m;x}^{jq^*_{n^*}}\varphi_{i^*_{n^*}}(x)\times\\\prod_{k=0}^{j-1}\dfrac{m_0(k(v_{\underline{i}^*,\underline{q}^*,\underline{r}^*}+\kappa-i^*_{n^*})+v_{\underline{i}^*,\underline{q}^*,\underline{r}^*})}{m_0(k(v_{\underline{i}^*,\underline{q}^*,\underline{r}^*}+\kappa-i^*_{n^*}))}+\text{rem}_j(x)
\end{multline*}
with $A_{\underline{i}^*,\underline{q}^*,\underline{r}^*}(0)>0$ and $\text{rem}_j(0)\geq0$. Observe that from Lemma \ref{ExFunction}, we have
$$\partial_{m;x}^{jq^*_{n^*}}\varphi_{i^*_{n^*}}(x)=\dprod{d=1}{N}{\left(\dsum{j_d\geq0}{}{a_d^{j_d+jq^*_{n^*,d}}(j_d+jq^*_{n^*,d})!^{s_d}\dfrac{x_d^{j_d}}}{m_d(j_d)}\right)};$$
hence,
$$\partial_{m;x}^{jq^*_{n^*}}\varphi_{i^*_{n^*}}(0)=\dprod{d=1}{N}a_d^{jq^*_{n^*,d}}(jq^*_{n^*,d})!^{s_d}.$$
Observe also that, since $m_0$ is a regular moment fuction of order $s_0$, there exists a positive constant $a_0>0$ such that
\begin{multline*}
\dfrac{m_0(k(v_{\underline{i}^*,\underline{q}^*,\underline{r}^*}+\kappa-i^*_{n^*})+v_{\underline{i}^*,\underline{q}^*,\underline{r}^*})}{m_0(k(v_{\underline{i}^*,\underline{q}^*,\underline{r}^*}+\kappa-i^*_{n^*}))}\\\geq a_0^{v_{\underline{i}^*,\underline{q}^*,\underline{r}^*}}\left(\dprod{\ell=1}{v_{\underline{i}^*,\underline{q}^*,\underline{r}^*}}{\left(k(v_{\underline{i}^*,\underline{q}^*,\underline{r}^*}+\kappa-i^*_{n^*})+\ell\right)}\right)^{s_0};\end{multline*}
hence;
\begin{multline*}
\prod_{k=0}^{j-1}\dfrac{m_0(k(v_{\underline{i}^*,\underline{q}^*,\underline{r}^*}+\kappa-i^*_{n^*})+v_{\underline{i}^*,\underline{q}^*,\underline{r}^*})}{m_0(k(v_{\underline{i}^*,\underline{q}^*,\underline{r}^*}+\kappa-i^*_{n^*}))}\\\geq
a_0^{v_{\underline{i}^*,\underline{q}^*,\underline{r}^*}}\left(\prod_{k=0}^{j-1}\dprod{\ell=1}{v_{\underline{i}^*,\underline{q}^*,\underline{r}^*}}{\left(k(v_{\underline{i}^*,\underline{q}^*,\underline{r}^*}+\kappa-i^*_{n^*})+\ell\right)}\right)^{s_0}.
\end{multline*}
Now, let us notice that Lemma \ref{TechnicalLemmaEx} implies that
$$
\prod_{k=0}^{j-1}\dprod{\ell=1}{v_{\underline{i}^*,\underline{q}^*,\underline{r}^*}}{\left(k(v_{\underline{i}^*,\underline{q}^*,\underline{r}^*}+\kappa-i^*_{n^*})+\ell\right)}\ge(jv_{\underline{i}^*,\underline{q}^*,\underline{r}^*})!
$$
Applying then this last inequality, we deduce that there exist two positive constants $C,K>0$ such that
\begin{equation}\label{ineqEx}
\TC{u}{j(v_{\underline{i}^*,\underline{q}^*,\underline{r}^*}+\kappa-i_{n^*}^*)+i^*_{n^*}}(0)\geq CK^j(jv_{\underline{i}^*,\underline{q}^*,\underline{r}^*})!^{s_0}\dprod{d=1}{N}(jq^*_{n^*,d})!^{s_d}.
\end{equation}

Let us now suppose that $\widetilde{u}(t,x)$ is $\sigma'$-Gevrey for some $\sigma'<\sigma_c$. Then, Definition \ref{defi:gevrey-order}, properties of moment functions and inequality (\ref{ineqEx}) imply
\begin{equation}\label{equiEx}
1\leq C'K'^j\frac{\Gamma(1+(\sigma'+s_0)(j(v_{\underline{i}^*,\underline{q}^*,\underline{r}^*}+\kappa-i^*_{n^*})+i^*_{n^*}))}{(jv_{\underline{i}^*,\underline{q}^*,\underline{r}^*})!^{s_0}\dprod{d=1}{N}(jq^*_{n^*,d})!^{s_d}}
\end{equation}
for all $j\geq0$ and some convenient positive constants $C',K'>0$ independent of $j$. Proposition \ref{ex} follows, since such inequalities are impossible from the Stirling's Formula and from the definition of $\sigma_c$ (see Definition \ref{DefSigmac}). Indeed, this tells us that the right hand-side of (\ref{equiEx}) goes to $0$ when $j$ tends to infinity. This ends the proof.
\end{proof}

\begin{lemma}\label{TechnicalLemmaEx}
Let $j\geq1$ and $v\geq0$. Then, for every integer $a\geq0$:
$$\prod_{k=0}^{j-1}\dprod{\ell=1}{v}{\left(k(v+a)+\ell\right)}\geq(jv)!.$$
\end{lemma}

\begin{proof}
The relation being obvious when $v=0$ (the second product is $1$), we assume $v\geq1$ and we proceed by induction on $j$.

The inequality is clear for $j=1$. Let us now suppose that it holds for a certain $j\geq1$. Then,
\begin{align*}
\prod_{k=0}^{j}\dprod{\ell=1}{v}{\left(k(v+a)+\ell\right)}&\geq(jv)!\dprod{\ell=1}{v}{\left(j(v+a)+\ell\right)}\\
&=\dfrac{((j+1)v+ja)!}{(jv+ja)!}(jv)!\\
&=\dfrac{\dbinom{(j+1)v+ja}{ja}}{\dbinom{jv+ja}{ja}}((j+1)v)!
\end{align*}
and the result follows since $\dbinom{(j+1)v+ja}{ja}\geq\dbinom{jv+ja}{ja}$.
\end{proof}

\subsection{Remark on the Cauchy-Kovalevskaya Theorem and directions for further research}
When the moment functions $m_0,m_1,...,m_N$ are chosen so that $m_0(\lambda)=m_1(\lambda)...=m_N(\lambda)=\Gamma(1+\lambda)$, Eq. (\ref{EQ1}) is reduced to a classical inhomogeneous nonlinear partial differential equation. In particular, our main Theorem \ref{GevreyRegularityTheorem} allows to study the Gevrey regularity of its formal power series solution, including the non-Kovalevskaya case.

However, in the Kovalevskaya case, it is important to note here that our result is weaker than the Cauchy-Kovalevskaya Theorem. Let us consider for instance the partial differential equation
\begin{equation}\label{EQExKC}
\begin{cases}
\partial_t^3u+\partial_t\partial_x u+(\partial_x^2u)^3=0\\
\partial_t^ju(t,x)|_{t=0}=\varphi_j(x),\  j=0,1,2
\end{cases}.
\end{equation}
in two variables $(t,x)\in\C^2$. Then, the Cauchy-Kovalevskaya Theorem tells us that the formal solution $\widetilde{u}(t,x)$ defines an analytic function at the origin of $\C^2$, whereas our Theorem \ref{GevreyRegularityTheorem} tells us that $\widetilde{u}(t,x)$ is $1$-Gevrey. This is not contradictory, of course, but our result is clearly weaker.

This is probably due to the choice of our Newton polygon and the calculation method we used. So, as directions for future research, it seems interesting to improve our result on the Gevrey order of the formal solution of Eq. (\ref{EQ1}).

\bibliographystyle{plain}
\bibliography{nonlinear_moment_PDE_V11}

\end{document}